\documentclass[11pt]{amsproc}

\usepackage{amssymb}

\usepackage{graphicx}
\usepackage{geometry}
\usepackage{amsmath}
\usepackage{tikz}
\usetikzlibrary{patterns,arrows.meta,calc}

\tikzset{
  cone/.style      = {draw=teal!75!black, line width=0.9pt},
  bandline/.style  = {draw=orange!90!yellow, line width=1.7pt},
  axisline/.style  = {draw=black, line width=0.7pt, dash pattern=on 2pt off 2pt},
  axissolid/.style = {draw=black, line width=0.7pt},
  axisdot/.style   = {draw=black, line width=0.7pt, line cap=round,
                      dash pattern=on 0.1pt off 2.2pt},
  edgeline/.style  = {draw=black, line width=1.1pt},
  edgedash/.style  = {draw=black, line width=1.1pt, dash pattern=on 2pt off 2pt},
  redline/.style   = {draw=red!65!black, line width=1.1pt},
  reddash/.style   = {draw=red!65!black, line width=1.1pt, dash pattern=on 2pt off 2pt},
  lbl/.style       = {font=\large},
  slbl/.style      = {font=\normalsize},
  blk/.style     = {draw=black, line width=1.2pt}, 
  red/.style     = {draw=red!70!black, line width=1.4pt},
  gold/.style    = {draw=orange!90!yellow, line width=1.5pt},
  tealarc/.style = {draw=cyan!70!teal, line width=1.5pt},
  bv/.style      = {circle, fill=black, inner sep=3.0pt},
  rv/.style      = {circle, fill=red!70!black, inner sep=3.0pt},
  rvbig/.style   = {circle, fill=red!70!black, inner sep=3.0pt},
  tv/.style      = {circle, draw=teal!75!black, fill=white, line width=0.8pt, inner sep=1.5pt},
  uv/.style      = {circle, fill=black, inner sep=2.4pt},
  urv/.style     = {circle, fill=red!70!black, inner sep=2.4pt},
  lbl/.style     = {font=\large},
  vlbl/.style    = {font=\normalsize}
}

\newcommand{\drawcone}[7]{%
  \pgfmathsetmacro{\Cx}{#1+#4*cos(#3)}%
  \pgfmathsetmacro{\Cy}{#2+#4*sin(#3)}%
  \pgfmathsetmacro{\Ux}{cos(#3)}%
  \pgfmathsetmacro{\Uy}{sin(#3)}%
  \pgfmathsetmacro{\Nx}{-sin(#3)}%
  \pgfmathsetmacro{\Ny}{cos(#3)}%
  \pgfmathsetmacro{\Rpx}{\Cx+#5*\Nx}\pgfmathsetmacro{\Rpy}{\Cy+#5*\Ny}%
  \pgfmathsetmacro{\Rqx}{\Cx-#5*\Nx}\pgfmathsetmacro{\Rqy}{\Cy-#5*\Ny}%
  \draw[cone] (#1,#2) -- (\Rpx,\Rpy);
  \draw[cone] (#1,#2) -- (\Rqx,\Rqy);
  \pgfmathsetmacro{\Sgn}{sin(#3)}%
  \ifdim\Sgn pt>0pt
    \draw[cone,#7] plot[domain=-90:90,samples=70,variable=\cp]
      ({\Cx+#6*cos(\cp)*\Ux+#5*sin(\cp)*\Nx},{\Cy+#6*cos(\cp)*\Uy+#5*sin(\cp)*\Ny});
    \draw[cone] plot[domain=90:270,samples=70,variable=\cp]
      ({\Cx+#6*cos(\cp)*\Ux+#5*sin(\cp)*\Nx},{\Cy+#6*cos(\cp)*\Uy+#5*sin(\cp)*\Ny});
  \else
    \draw[cone,#7] plot[domain=90:270,samples=70,variable=\cp]
      ({\Cx+#6*cos(\cp)*\Ux+#5*sin(\cp)*\Nx},{\Cy+#6*cos(\cp)*\Uy+#5*sin(\cp)*\Ny});
    \draw[cone] plot[domain=-90:90,samples=70,variable=\cp]
      ({\Cx+#6*cos(\cp)*\Ux+#5*sin(\cp)*\Nx},{\Cy+#6*cos(\cp)*\Uy+#5*sin(\cp)*\Ny});
  \fi
}

\newcommand{\drawconeT}[7]{%
  \pgfmathsetmacro{\Ux}{cos(#3)}\pgfmathsetmacro{\Uy}{sin(#3)}%
  \pgfmathsetmacro{\Nx}{-sin(#3)}\pgfmathsetmacro{\Ny}{cos(#3)}%
  \pgfmathsetmacro{\Cx}{#1+#4*\Ux}\pgfmathsetmacro{\Cy}{#2+#4*\Uy}%
  \pgfmathsetmacro{\Ecc}{#6/#4}%
  \pgfmathsetmacro{\Xz}{#4-#6*#6/#4}%
  \pgfmathsetmacro{\Yz}{#5*sqrt(1-\Ecc*\Ecc)}%
  \draw[cone] (#1,#2) -- ({#1+\Xz*\Ux+\Yz*\Nx},{#2+\Xz*\Uy+\Yz*\Ny});
  \draw[cone] (#1,#2) -- ({#1+\Xz*\Ux-\Yz*\Nx},{#2+\Xz*\Uy-\Yz*\Ny});
  \draw[cone,#7] plot[domain=90:270,samples=90,variable=\cp]
    ({\Cx+#6*cos(\cp)*\Ux+#5*sin(\cp)*\Nx},{\Cy+#6*cos(\cp)*\Uy+#5*sin(\cp)*\Ny});
  \draw[cone] plot[domain=-90:90,samples=90,variable=\cp]
    ({\Cx+#6*cos(\cp)*\Ux+#5*sin(\cp)*\Nx},{\Cy+#6*cos(\cp)*\Uy+#5*sin(\cp)*\Ny});
}

\newcommand{\coneaxisup}[5]{%
  \pgfmathsetmacro{\La}{#4-#5}%
  \pgfmathsetmacro{\Ux}{cos(#3)}\pgfmathsetmacro{\Uy}{sin(#3)}%
  \draw[axissolid] (#1,#2) -- ({#1+0.6667*\La*\Ux},{#2+0.6667*\La*\Uy});
  \draw[axisline] ({#1+0.6667*\La*\Ux},{#2+0.6667*\La*\Uy})
               -- ({#1+\La*\Ux},{#2+\La*\Uy});
}

\newcommand{\coneray}[7]{%
  \pgfmathsetmacro{\Ux}{cos(#3)}\pgfmathsetmacro{\Uy}{sin(#3)}%
  \pgfmathsetmacro{\Nx}{-sin(#3)}\pgfmathsetmacro{\Ny}{cos(#3)}%
  \pgfmathsetmacro{\Tx}{#1+#4*\Ux+#6*cos(#7)*\Ux+#5*sin(#7)*\Nx}%
  \pgfmathsetmacro{\Ty}{#2+#4*\Uy+#6*cos(#7)*\Uy+#5*sin(#7)*\Ny}%
  \draw[axissolid] (#1,#2) -- ({#1+0.6667*(\Tx-#1)},{#2+0.6667*(\Ty-#2)});
  \draw[axisline] ({#1+0.6667*(\Tx-#1)},{#2+0.6667*(\Ty-#2)}) -- (\Tx,\Ty);
}

\newcommand{\coneaxisdown}[7]{%
  \pgfmathsetmacro{\Ux}{cos(#3)}\pgfmathsetmacro{\Uy}{sin(#3)}%
  \draw[axisdot]   (#1,#2) -- ({#1+#7*\Ux},{#2+#7*\Uy});
  \draw[axissolid] ({#1+#7*\Ux},{#2+#7*\Uy}) -- ({#1+#4*\Ux},{#2+#4*\Uy});
  \draw[axisline]  ({#1+#4*\Ux},{#2+#4*\Uy})
                -- ({#1+(#4+#6)*\Ux},{#2+(#4+#6)*\Uy});
}

\newcommand{\narrowgoldband}{%
  \pattern[pattern=finegrid, pattern color=teal!75!black, opacity=0.6]
    plot[domain=-2.0:2.0,samples=90,variable=\xp] ({\xp},{0.9875-0.055*\xp*\xp})
    -- plot[domain=2.0:-2.0,samples=90,variable=\xp] ({\xp},{0.1625-0.055*\xp*\xp})
    -- cycle;
}

\newcommand{\maskcone}[6]{%
  \pgfmathsetmacro{\Cx}{#1+#4*cos(#3)}%
  \pgfmathsetmacro{\Ux}{cos(#3)}%
  \pgfmathsetmacro{\Uy}{sin(#3)}%
  \pgfmathsetmacro{\Nx}{-sin(#3)}%
  \pgfmathsetmacro{\Ny}{cos(#3)}%
  \pgfmathsetmacro{\Cy}{#2+#4*sin(#3)}%
  \pgfmathsetmacro{\Sgn}{sin(#3)}%
  \ifdim\Sgn pt>0pt
    \fill[white] (#1,#2) -- plot[domain=-90:90,samples=70,variable=\cp]
      ({\Cx+#6*cos(\cp)*\Ux+#5*sin(\cp)*\Nx},{\Cy+#6*cos(\cp)*\Uy+#5*sin(\cp)*\Ny}) -- cycle;
  \else
    \fill[white] (#1,#2) -- plot[domain=90:270,samples=70,variable=\cp]
      ({\Cx+#6*cos(\cp)*\Ux+#5*sin(\cp)*\Nx},{\Cy+#6*cos(\cp)*\Uy+#5*sin(\cp)*\Ny}) -- cycle;
  \fi
}

\newcommand{\narrowgoldbanddeep}{%
  \pattern[pattern=finegrid, pattern color=teal!75!black, opacity=0.6]
    plot[domain=-2.0:2.0,samples=90,variable=\xp] ({\xp},{1.1971-0.055*\xp*\xp})
    -- plot[domain=2.0:-2.0,samples=90,variable=\xp] ({\xp},{0.0965-0.1436-0.055*\xp*\xp})
    -- cycle;
}

\newcommand{\stubto}[4]{%
  \pgfmathsetmacro{\ux}{cos(#3)}\pgfmathsetmacro{\uy}{sin(#3)}%
  \pgfmathsetmacro{\tt}{(#4-0.055*(#1)*(#1)-(#2))/(\uy+0.11*(#1)*\ux)}%
  \pgfmathsetmacro{\tt}{\tt-((#2)+\uy*\tt-(#4)+0.055*((#1)+\ux*\tt)*((#1)+\ux*\tt))
                            /(\uy+0.11*((#1)+\ux*\tt)*\ux)}%
  \pgfmathsetmacro{\tt}{\tt-((#2)+\uy*\tt-(#4)+0.055*((#1)+\ux*\tt)*((#1)+\ux*\tt))
                            /(\uy+0.11*((#1)+\ux*\tt)*\ux)}%
  \pgfmathsetmacro{\Ex}{(#1)+\ux*\tt}\pgfmathsetmacro{\Ey}{(#2)+\uy*\tt}%
  \draw[edgeline] (#1,#2) -- ({(#1)+0.6667*(\Ex-(#1))},{(#2)+0.6667*(\Ey-(#2))});
  \draw[edgedash] ({(#1)+0.6667*(\Ex-(#1))},{(#2)+0.6667*(\Ey-(#2))}) -- (\Ex,\Ey);
}

\newcommand{\blackframe}[1]{%
  \pgfmathsetmacro{\Tx}{-0.58+#1}\pgfmathsetmacro{\Bx}{-0.62+#1}%
  \stubto{\Tx}{0.88}{118}{1.1971}
  \stubto{\Tx}{0.88}{62}{1.1971}
  \stubto{\Bx}{0.28}{242}{-0.0471}
  \stubto{\Bx}{0.28}{270}{-0.0471}
  \stubto{\Bx}{0.28}{298}{-0.0471}
  \fill (\Tx,0.88) circle (2.3pt);
  \fill (\Bx,0.28) circle (2.3pt);
}

\newcommand{\coneslice}[5]{%
  \draw[redline] plot[domain=180:360,samples=90,variable=\cp]
    ({#1+#4*cos(\cp)*cos(#3)-#5*sin(\cp)*sin(#3)},
     {#2+#4*cos(\cp)*sin(#3)+#5*sin(\cp)*cos(#3)});
  \draw[reddash] plot[domain=0:180,samples=90,variable=\cp]
    ({#1+#4*cos(\cp)*cos(#3)-#5*sin(\cp)*sin(#3)},
     {#2+#4*cos(\cp)*sin(#3)+#5*sin(\cp)*cos(#3)});
}

\pgfdeclarepatternformonly{finegrid}{\pgfqpoint{-1pt}{-1pt}}{\pgfqpoint{4pt}{4pt}}{\pgfqpoint{3pt}{3pt}}%
{%
  \pgfsetlinewidth{0.15pt}%
  \pgfpathmoveto{\pgfqpoint{0pt}{0pt}}%
  \pgfpathlineto{\pgfqpoint{0pt}{3.1pt}}%
  \pgfpathmoveto{\pgfqpoint{0pt}{0pt}}%
  \pgfpathlineto{\pgfqpoint{3.1pt}{0pt}}%
  \pgfusepath{stroke}%
}

\usepackage[colorlinks=true,allcolors=blue]{hyperref}
\usepackage[shortlabels]{enumitem}

\usepackage[dvipsnames]{xcolor}

\usepackage[only,sslash,bbslash]{stmaryrd}%

\DeclareFontFamily{U}  {MnSymbolC}{}
\DeclareSymbolFont{MnSyC}         {U}  {MnSymbolC}{m}{n}
\SetSymbolFont{MnSyC}       {bold}{U}  {MnSymbolC}{b}{n}
\DeclareFontShape{U}{MnSymbolC}{m}{n}{
   <-6>  MnSymbolC5
  <6-7>  MnSymbolC6
  <7-8>  MnSymbolC7
  <8-9>  MnSymbolC8
  <9-10> MnSymbolC9
 <10-12> MnSymbolC10
 <12->   MnSymbolC12}{}
\DeclareFontShape{U}{MnSymbolC}{b}{n}{
   <-6>  MnSymbolC-Bold5
  <6-7>  MnSymbolC-Bold6
  <7-8>  MnSymbolC-Bold7
  <8-9>  MnSymbolC-Bold8
  <9-10> MnSymbolC-Bold9
 <10-12> MnSymbolC-Bold10
 <12->   MnSymbolC-Bold12}{}
\DeclareMathSymbol{\mnbowtie}{\mathbin}{MnSyC}{38}
\DeclareMathSymbol{\mnvertbowtie}{\mathbin}{MnSyC}{39}

\newtheorem{theorem}{Theorem}[section]
\newtheorem{lemma}[theorem]{Lemma}

\newtheorem{corollary}[theorem]{Corollary}

\theoremstyle{definition}
\newtheorem{definition}[theorem]{Definition}

\newtheorem{observation}[theorem]{Observation}

\theoremstyle{remark}

\numberwithin{equation}{section}

\long\def\ignore#1{}

\def\siran#1{\v{S}ir\'{a}\v{n}}
\def\skoviera#1{\v{S}koviera}
\def\fijavz#1{Fijav\v{z}}
\def\bollobas#1{Bollob\'{a}s}
\def\mobius#1{M\"{o}bius}

\def\mR{\mathbb{R}}

\let\Ga\Gamma
\let\Si\Sigma
\let\Om\Omega
\let\al\alpha
\let\be\beta
\let\ga\gamma
\let\de\delta

\def\ir{^\circ}

\let\sb\subseteq
\def\eg{{\fam1 eg}}
\let\ti\widetilde
\let\bdy\partial

\def\iv{^{-1}}
\def\mapright#1{
	\;\smash{\mathop{\longrightarrow}\limits^{#1}}\;}

\def\pdu{{*}}
\def\ppe{{\times}}
\def\pwi{{\odot}}
\def\du{{}^*}
\def\pe{{}^\times}
\def\wi{{}^\odot}

\def\ct{{/}}
\def\dt{{\backslash}}
\def\tc{\mskip-2mu\rightthreetimes\mskip-2mu}

\def\er{\ominus}

\def\pn{\mskip-2mu\mnbowtie\mskip-2mu}
\def\dpn{\mnvertbowtie}

\def\sct{\ct\mskip-4mu\ct}
\def\sdt{\dt\mskip-4mu\dt}
\def\sw{\mskip-2mu\sim\mskip-2mu}

\def\ms#1{\mskip#1mu}

\def\ca{\mathsf{a}}
\def\cv{\mathsf{v}}
\def\cf{\mathsf{f}}
\def\cz{\mathsf{z}}
\def\cvv{\cv^+}
\def\cff{\cf^+}

\def\xp#1{\par\global\leftskip=#1\parindent
	\noindent\hskip-\parindent\ignorespaces}
\def\pcap#1#2{{\llap{\hbox to #1\parindent{#2\hss}}}\ignorespaces}
\begin{document}

\let\reallabel\label

\title[Embeddings in pseudosurfaces]{Duality and minors for embeddings of graphs in pseudosurfaces}


\author[Dunshee]{Blake Dunshee}
\address{Department of Mathematics and Computer Science, Belmont
University, 1900 Belmont Boulevard, Nashville, Tennessee 37212}
\curraddr{}
\email{blake.dunshee@belmont.edu}
\thanks{}

\author[Ellingham]{M.~N.~Ellingham}
\address{Department of Mathematics, 1326 Stevenson Center, Vanderbilt
University, Nashville, Tennessee 37240}
\curraddr{}
\email{mark.ellingham@vanderbilt.edu}
\thanks{}

\author[Ellis-Monaghan]{Joanna A.~Ellis-Monaghan}
\address{Korteweg-de Vries Institute for Mathematics, University of
Amsterdam, Science Park 105-107, 1098 XH Amsterdam, the Netherlands}
\curraddr{}
\email{jellismonaghan@gmail.com}
\urladdr{https://sites.google.com/site/joellismonaghan/}
\thanks{}


\dedicatory{Date: 17 September 2026}
\date{}

\begin{abstract}
Cellular embeddings of graphs in surfaces have well-defined duality and minor (edge contraction and deletion) operations that interact in a natural way.
A \emph{pseudosurface} is obtained from a surface (compact $2$-manifold) by a finite number of identifications of finite sets of points.
Points that are created by the identifications do not have a neighborhood homeomorphic to an open disk and are known as \emph{pinchpoints}.
Embeddings of graphs in pseudosurfaces have been considered, both implicitly and explicitly, since the 1960s.
Usually the condition that all pinchpoints correspond to vertices of the graph is imposed.
However, this makes it difficult to define duality and minors for pseudosurface embeddings and have these operations interact in the expected way.

In 1991 Deneen, Shute, and Thomborson (DST) proposed two models for duality of pseudosurface embeddings, 
the more general of which was rediscovered and described in two ways by Huggett and Moffatt (HM) in 2020.  Huggett and Moffatt also implemented minor operations in this framework.
Both the more general DST model and the two HM models involved extending a class of graph embeddings in pseudosurfaces to a larger class of somewhat abstract objects.
We define the class of \emph{pseudocellular} embeddings of graphs in pseudosurfaces,  which allow pinchpoints at places other than vertices, in particular in the middle of faces or edges.
The DST/HM objects are represented by a subclass of pseudocellular embeddings called \emph{quasicellular embeddings}, giving them a simple topological interpretation.

Pseudocellular embeddings also generalize other structures, including the \emph{edge-point ribbon graphs} of Ellis-Monaghan, Kauffman, and Moffatt, and \emph{cyclically ordered graphs} or \emph{cogs} (also known as \emph{rigid-vertex graphs}).
Duality and minor operations for pseudocellular embeddings have very simple and straightforward definitions using topological quotient operations.
Pseudocellular embeddings of edgeless graphs have nontrivial structure, and we define some minor operations for those that are related to `$t$-minor' operations on bipartite graphs.
We develop a family of polynomial invariants for pseudocellular embeddings and discuss connections to other polynomial invariants.

\end{abstract}

\maketitle

\section{Introduction}\label{sec:intro}

While embeddings of graphs in surfaces have arguably been studied since the ancient Greeks considered the Platonic polyhedra, embeddings of graphs in pseudosurfaces seem to have only been studied since the 1960s.
By a pseudosurface we will mean a topological space obtained from a (not necessarily connected) compact surface by identifying together finitely many finite sets of points.  The new points are called pinchpoints, and the pinchdegree of such a point is the number of corresponding points in the original surface.
Formal definitions in terms of quotient spaces are given in Section \ref{sec:topology}.

Embeddings of graphs in (compact) surfaces have a very nice theory involving duality and minor operations such as deletion and contraction of edges. This has been extended over the past twenty years or so to incorporate ideas such as Chmutov's partial duality \cite{Chm09} and the more general twisted duality framework by one of us (J.E.-M.) and Moffatt \cite{EM12}.  
This theory generally applies to cellular embeddings, where all faces are homeomorphic to open disks.
It is desirable to extend these ideas, as far as possible, to a natural class of embeddings of graphs in pseudosurfaces.

Our main objective in this paper is to define a new class of embeddings of graphs in pseudosurfaces known as pseudocellular embeddings, for which duality and minor operations can be defined in very natural ways.
We first define a smaller class of embeddings, quasicellular embeddings, which provide a simple topological interpretation of some more complicated models of duality and minors that already appear in the literature \cite{DST91,HM20}.
The difficulty with many previous approaches to duality for graph embeddings in pseudosurfaces is that they require pinchpoints to be occupied by vertices of the graph.
Quasicellular embeddings allow each face to contain one pinchpoint, and pseudocellular embeddings additionally allow the interior of each edge to contain one pinchpoint.
The dual of a pseudocellular embedding is embedded in the same pseudosurface as the primal embedding.
Quasicellular and pseudocellular embeddings are defined in Subsections \ref{ss:quasicellular} and \ref{ss:pseudocellular}, respectively.

We develop some background before defining pseudocellular embeddings.
In Section \ref{sec:topology} we provide a firm topological foundation for the minor and other operations we define for pseudocellular embeddings.
In Section \ref{sec:graphemb} we summarize some standard approaches to duality, minors, and related topics for graphs embedded in surfaces.
To put our main objective into perspective, we give a detailed discussion of previous models of duality and minors for graphs embedded in pseudosurfaces in Section \ref{sec:embpseudosurf}.
Our minor operations are based on topological quotienting, and in Subsection \ref{ss:qbr} (and in later parts of the paper) we contrast these Q (quotient-based) minor operations with the usual BR (\bollobas.-Riordan) minor operations for graphs embedded in surfaces \cite{BR02}, discussed in Subsection \ref{ss:minorsurf}.  We give criteria for which type of operation, Q or BR, should be applied in general settings.

After defining pseudocellular embeddings we explore some consequences.
In Subsection \ref{ss:subclass} we discuss subclasses of pseudocellular embeddings, some of which correspond to some other objects in the literature, including edge-point ribbon graphs \cite{EKM18} and cyclically ordered graphs, also known as cogs or rigid-vertex graphs (see for example \cite{BEEMM25plus,BDJSV15,EM13,Kau89}).
We consider additional operations beyond edge contraction and deletion in Section \ref{sec:addop}, including partial Petrie and Wilson duality, operations to delete and contract vertices and faces, and additional operations that eliminate edges.  (We use \emph{elimination} as a general term for operations that get rid of a vertex, edge, or face in some way.)
In Section \ref{sec:poly} we define a family of polynomial invariants for pseudocellular embeddings, and discuss connections to other polynomial invariants.
Finally, in Section \ref{sec:conclusion} we provide a table summarizing the operations we have defined for pseudocellular embeddings, and mention some possible further directions suggested by our framework.

\section{Topological background}\label{sec:topology}

In this section we discuss the topology of surfaces and pseudosurfaces, proving some results we will need later for defining various operations on embeddings of graphs in pseudosurfaces.
Previous papers on embeddings of graphs in pseudosurfaces have usually taken a fairly informal approach, and avoided detailed discussion of topological considerations.  Here we try to provide a firmer topological foundation, although we do not claim to be completely rigorous.  
We will later rely on a paper of Youngs \cite{You63} for some important topological ideas, such as cellularization.
This section culminates in Theorem \ref{curvecontract}, which we will use to show that our operations preserve the fact that we have a pseudosurface.

\subsection{Quotient spaces in topology}\label{ss:quotient}
We will use the quotient operation for topological spaces in several ways in this paper.  Pseudosurfaces will be defined as quotients of surfaces.    We will represent graphs topologically using $1$-complexes, which are quotient spaces, and we will also define minor operations for embeddings using quotients.

We assume the reader is familiar with basic topological concepts.
We represent a topological space as an ordered pair $(X, \top_X)$, where $X$ is a set of points and $\top_X$ is a topology for $X$, i.e., the collection of open sets.  A \emph{neighborhood} of $x \in X$ (or $S \sb X$) is a set containing an open set that includes $x$ (or $S$).  The interior of $X$ is denoted $X\ir$.  Continuity of a function $f: X \to Y$ between two topological spaces can be defined pointwise at each $x \in X$: $f$ is continuous at $x$ if for each neighborhood $B$ of $f(x)$, $f\iv(B)$ is a neighborhood of $x$.
Then $f$ is continuous if and only if it is continuous at each $x \in X$.  By a \emph{component} of $X$ we always mean a connected component.

The fundamental concept of a quotient topology can be approached from two slightly different perspectives.
First, suppose topological spaces $(X, \top_X)$ and $(Y, \top_Y)$ are given.  A \emph{quotient map} from $(X, \top_X)$ to $(Y, \top_Y)$ is a function $q : X \to Y$ such that $q$ is a surjection, and $B \sb Y$ is open in $\top_Y$ if and only if $q\iv(B)$ is open in $\top_X$.
A quotient map is always continuous.
Second, suppose $(X, \top_X)$ and a surjection $q: X \to Y$ are given, where $Y$ is any set.
Then there is a unique topology on $Y$ that makes $q$ a quotient map, which we call the \emph{quotient topology induced by $q$}, denoted $\top_X/q$.  The open sets in $\top_X/q$ are precisely those $B \sb Y$ for which $q\iv(B)$ is open in $\top_X$.

Given a function $f : X \to Y$, the \emph{equivalence relation induced by $f$} is the relation $=_f$ on $X$, where $x =_f x'$ if and only if $f(x) = f(x')$.
A point $x \in X$ is a \emph{solo input} of $f$, and $f(x)$ is a \emph{solo output} of $f$, if $x$ is equivalent only to itself under $=_f$, \emph{i.e.}, $f(x') \ne f(x)$ for all $x' \ne x$.
Otherwise $x$ is a \emph{merging input} of $f$, and $f(x)$ is a \emph{merged output}.

Two fundamental properties of quotient spaces are as follows.

\begin{lemma}\label{qp}
Let $(X, \top_X)$ be a topological space, and suppose there are surjections
$$ X = Y_0 \mapright{q_1} Y_1
	\mapright{q_2} Y_2 \;\cdots\; Y_{k-1}
	\mapright{q_k} Y_k = Y.$$
\begin{enumerate}[label=\textup{(\arabic*)},itemsep=2pt]
\item
Then the topologies $\top_X/q_1/q_2/\dots/q_k$ and $\top_X/(q_k \circ \dots \circ q_2 \circ q_1)$ on $Y$ are identical.

\item Suppose there is another sequence of surjections
$$ X = Z_0 \mapright{r_1} Z_1
	\mapright{r_2} Z_2 \;\cdots\; Z_{\ell-1}
	\mapright{r_\ell} Z_\ell = Z,$$
where $q = q_k \circ \dots \circ q_2 \circ q_1$
and $r = r_\ell \circ \dots \circ r_2 \circ r_1$ induce the same equivalence relation on $X$.
Then there is a homeomorphism $h$ from $(Y, \top_X/q)$ to $(Z, \top_X/r)$, such that $r = h \circ q$, or equivalently $q\iv(y) = r\iv(h(y))$ for all $y \in Y$.
\end{enumerate}
\end{lemma}

\begin{proof}
Part (1) follows from the fact that the composition of two quotient maps is a quotient map (see, e.g., \cite[p.~35]{ST80} or \cite[Proposition~3.62]{Lee11}) and the uniqueness of the quotient topology.  For part (2) see, e.g., \cite[Theorem~3.75]{Lee11}.
\end{proof}

We are interested in situations where two sets $X$ and $Y$ are `mostly the same' in some sense, and it will simplify our exposition if we can regard most points in $X$ as also being present in $Y$.
Therefore, we define an \emph{identification map} from $X$ to $Y$ as a surjection $q : X \to Y$ where $q$ acts as the identity map on the set $X_1$ of solo inputs of $q$.
So the set of solo outputs of $q$ is also $X_1$, and $q$ combines the merging inputs in $X_2=X-X_1$ into the merged outputs in $Y_2=Y-X_1$.
When we talk about `identifying points' below, formally we mean that there is an identification map.
Since an identification map $q : X \to Y$ is a surjection, it gives a quotient topology $\top_X/q$ on $Y$ for a given topology $\top_X$ on $X$.  We say an identification map is \emph{finite} if the set of merging inputs (and hence the set of merged outputs) is finite.

So the consequences of Lemma \ref{qp} are as follows.  If we form a new topological space $(Y, \top_Y)$ by identifying points in $(X, \top_X)$, we can break the identification process up into steps and determine the changes in the topology at each step.  We can change the order in which we identify points, as long as the overall identification maps induce the same equivalence relation on $X$.
The final topological space we arrive at will be unique up to renaming the merged outputs in $Y$.

We frequently use this with a `lifting' step.  For example, suppose we have identification maps $q_1: X \to Y$ and $q_2: Y \to Z$, where $X$ is a space with simple structure (e.g., a surface), $Y$ is a more complicated space (e.g., a pseudosurface), and we wish to describe the space $Z=Y/q_2$.  We can do this if we can find two other identification maps $r_1: X \to Y'$ and $r_2 : Y' \to Z$ where the effects of $r_1$ and $r_2$ are easy to describe and $r_2 \circ r_1 = q_2 \circ q_1$, because $Y/q_2 = (X/r_1)/r_2$.  So we lift from $Y$ to $X$ in order to compute the quotient space $Y/q_2$.

\begin{observation}[Closed inputs, finite outputs]\label{cifo}
In most cases where we will apply an identification map from $X$ to $Y$, the set of merged outputs will be finite, and the preimage of each merged output will be a closed set.
We discuss how to determine the topology on $Y$ in this case.

A topology $\top_X$ on a set $X$ can be described by giving an \emph{open neighborhood base} at every $x \in X$, namely a collection of open sets $\mathcal{B}_x$ such that every open set $S$ containing $x$ has an element of $\mathcal{B}_x$ as a subset.  The elements of $\mathcal{B}_x$ can be chosen to be `small'; in particular, we can always replace each $\mathcal{B}_x$ by the subset $\{B \in \mathcal{B}_x \;|\; B \sb A_x\}$ where $A_x$ is an arbitrary open set containing $x$.

Let $q : X \to Y$ be an identification map, where $X$ has topology $\top_X$.  Let $X_1$, $X_2 = X-X_1$, and $Y_2=Y-X_1$ be the sets of solo inputs (and also solo outputs), merging inputs, and merged outputs, respectively.  Suppose that $X_2$ is closed.
Then $X_1$ is open, so each $x \in X_1$ has an open neighborhood base in $\top_X$ consisting only of subsets of $X_1$; this is also an open neighborhood base for $x$ in $\top_Y = \top_X/q$.
So in some sense the topology does not change around these points.
Thus, to determine $\top_Y$ we need only determine an open neighborhood base for each $y \in Y_2$.

Suppose more specifically that $Y_2$ is finite and $q\iv(z)$ is closed for each $z \in Y_2$, so that $X_2 = \bigcup_{z \in Y_2} q\iv(z)$ is closed.
If $y \in Y_2$, then $X_2 - q\iv(y) = \bigcup_{z \in Y_2, z \ne y} q\iv(z)$ will also be closed, so its complement $X_1 \cup q\iv(y)$ will be open, and hence we can find an open neighborhood base for $y$ by taking the image under $q$ of `small' open sets $B$ in $X$ satisfying $q\iv(y) \sb B \sb X_1 \cup q\iv(y)$.
In particular, we can avoid the elements of $X_2 - q\iv(y)$ in determining the topology around $y$.
\end{observation}

One special type of identification map is where we identify a set $A \sb X$ to a single merged output $a$.
We call this \emph{crushing} $A$ to $a$.
An identification map is finite if and only if it is a composition of finite crushings.

\subsection{Curves}\label{ss:curve}
A \emph{curve} in a topological space $X$ is a continuous map $\ga : J \to X$ where $J \sb \mR$ is an interval with more than one element.
We say $\ga$ is \emph{simple} if it is injective.
By a \emph{simple path} we mean a simple curve whose domain is an interval $[a,b]$ (we are using `path' here in its topological sense).
It is well known that a simple path in a Hausdorff space is a homeomorphism from its domain to its range.
A curve $\ga : [a,b] \to X$ is \emph{closed} if $\ga(a) = \ga(b)$, and in this case it is a \emph{simple closed curve} if it is injective on $[a,b)$ (or, equivalently, on $(a,b]$).

Although a closed curve is sometimes defined as a continuous injective map from a circle, closed curves for us have a definite beginning and end.  We can change the beginning and end as follows.
A \emph{cyclic reparameterization} of a closed curve $\ga:[a,b] \to X$ is a curve $\ga_c: [a,b] \to X$ where $c \in [0, b-a)$, $\ga_c(x) = \ga(x+c)$ for $x \in [a, b-c]$, and $\ga_c(x) = \ga(x+a-b+c)$ for $x \in [b-c, b]$.  In other words, we change the starting point of $\ga$ from $\ga(a)$ to $\ga(a+c)$.

If we have a curve $\ga:(a, b) \to X$ where $X$ is a Hausdorff space, then we say that $\lim_{t \to a^+} \ga(t) = x$ for some $x \in X$ if for every open neighborhood $U$ of $x$ there is $\de > 0$ such that $\ga((a, a+\de)) \sb U$.
Since $X$ is Hausdorff this limit is unique if it exists, and by defining $\ga(a) = x$ we extend $\ga$ so it is continuous at $a$.
We can similarly define $\lim_{t \to b^-} \ga(t)$ and use it to make $\ga$ continuous at $b$.

\subsection{Surfaces}\label{ss:surf}
In this paper a \emph{surface} is a compact, but not necessarily connected, $2$-manifold.  As a $2$-manifold it is second countable and Hausdorff, and every point has an open neighborhood base of sets homeomorphic to open disks.  Compactness implies that there are finitely many components.  By the standard Classification of Surfaces Theorem each component is either a sphere with $h \ge 0$ handles added, denoted $S_h$ (an \emph{orientable} surface), or a sphere with $k \ge 1$ crosscaps added, denoted $N_k$ (a \emph{nonorientable} surface).
An orientable surface is one that can be given a clockwise direction that is globally consistent; an orientable surface equipped with such a clockwise direction is an \emph{oriented} surface.
The \emph{Euler genus} of a connected surface is either $\eg(S_h) = 2h$, or $\eg(N_k) = k$.

Suppose $\ga$ is a simple closed curve in a surface $\Si$, with image $C \sb \Si$.
A standard result in topology says that there is an open (regular) neighborhood of $C$ that is either an open \mobius. strip or an open cylinder, and we say that $\ga$ is a \emph{$1$-sided} or \emph{$2$-sided} curve, respectively.

\subsection{Pseudosurfaces}\label{ss:pseudosurf}
An \emph{open $m$-disk}, $m \ge 1$, is obtained from $m$ copies of an open disk by identifying their centers (i.e., it is the wedge sum of $m$ open disks).  The identified point is the \emph{center} of the $m$-disk.
When $m=1$ we designate an arbitrary point of the open $1$-disk as its center.
Closed $m$-disks are defined analogously.  Figure \ref{fig:3disk} shows a closed $3$-disk.  We refer to any $m$-disk as a \emph{multidisk}.

\begin{figure}[ht]
\hbox to\hsize{%
  \hss%

\begin{tikzpicture}[scale=1.25]
  \drawcone{0}{0}{90}{1.50}{0.50}{0.18}{solid}
  \drawcone{0}{0}{205}{1.35}{0.45}{0.17}{dashed}
  \drawcone{0}{0}{345}{1.45}{0.48}{0.18}{dashed}
\end{tikzpicture}
  \hss%
}
\caption{A $3$-disk}\label{fig:3disk}
\end{figure}
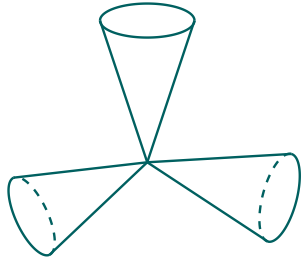

Suppose $\Si$ is a (not necessarily compact) $2$-manifold, and $\Pi$ is a quotient space $\Si/q$, where $q$ is a finite identification map $q : \Si \to \Pi$ (so there are finitely many merging inputs and merged outputs).
We say \emph{$\Pi$ is created by $q : \Si \to \Pi$}.  If $\Si$ is a surface we say $\Pi$ is a \emph{pseudosurface}.
Because $\Si$ is second countable and Hausdorff and we have finitely many merging inputs, it is easy to show that $\Pi$ is also second countable and Hausdorff.
For any $x \in \Pi$, the elements of $q\iv(x)$ are called the \emph{parents} of $x$, and $|q\iv(x)|$ is called the \emph{pinchdegree} of $x$.
By Observation \ref{cifo}, each solo output $r$ of $q$ has pinchdegree $1$, and still has an open neighborhood base of open disks.  We call $r$ a \emph{regular point}.
Each merged output $p$ has pinchdegree $m \ge 2$, and has an open neighborhood base of open $m$-disks centered at $p$.  We call $p$ a \emph{pinchpoint}.
Henceforth `multidisk (or $m$-disk) neighborhood of $x$' means `multidisk (or $m$-disk) neighborhood of $x$ centered at $x$' unless explicitly indicated otherwise.

\begin{observation}\label{quotps}
The composition of two finite identification maps is a finite identification map.
Therefore, a quotient space obtained from a pseudosurface by a finite identification map is also a pseudosurface.
\end{observation}

Given a space $\Pi$ created by $q: \Si \to \Pi$ as above, a \emph{pinchcomponent} of $\Pi$ is $\Pi'=q(\Si')$ where $\Si'$ is a component of $\Si$, which we call the \emph{parent component} of $\Pi'$.
Equivalently, a pinchcomponent is the closure in $\Pi$ of a component of $\Pi-P$, where $P$ is the set of pinchpoints.  Pinchcomponents may intersect at pinchpoints.
We have $\Pi'=\Si'/q'$, where $q'$ is $q$ restricted to domain $\Si'$ and range $\Pi'$.  Thus, if $\Pi$ is a pseudosurface, so that $\Si$ is a surface, a pinchcomponent $\Pi'$ is itself a pseudosurface with connected parent surface $\Si'$.

Under the usual definition of a pseudomanifold (see, e.g., \cite[Section 24]{ST80}) the parent surface of a pseudosurface must be a connected surface, i.e., there must be only one pinchcomponent.  The idea of a \emph{generalized pseudosurface}, whose parent surface may be disconnected, seems to have been introduced in the context of graph embeddings by White \cite{Whi78}.  Recent literature in graph theory generally uses `pseudosurface' with this generalized meaning, although older literature requires the parent surface to be connected.  We follow the more recent usage; the parent surfaces of our pseudosurfaces may be disconnected.

A pseudosurface $\Pi$ may equivalently be defined as a compact second countable Hausdorff space where every point $x$ has an open neighborhood that is an open multidisk (centered at $x$).  This is the approach used in \cite{DST91}.
Compactness implies that the set $P$ of points where the neighborhood is an $m$-disk for some $m \ge 2$ is finite.  If $P$ were infinite it would have an accumulation point $x$.
Every open neighborhood of $x$ would contain a point $p \in P$, $p \ne x$.  But $x$ has an open multidisk neighborhood, and in that neighborhood no point except possibly $x$ belongs to $P$, which is a contradiction.
Since $P$ is finite, we can find a parent surface $\Si$ for $\Pi$ by \emph{dismantling} each pinchpoint $p$, meaning we replace a closed $m$-disk neighborhood of $p$ containing no other pinchpoint by $m$ closed disks, each with a center consisting of a copy of $p$.
(Huggett and Moffatt \cite{HM20} call this \emph{resolving} a pinchpoint.)

A pseudosurface is \emph{orientable} if it can be given a clockwise direction at all regular points that is globally consistent; an orientable pseudosurface equipped with such a clockwise direction is an \emph{oriented} pseudosurface.
Orientability of a pseudosurface is equivalent to orientability of its parent surface.

Suppose $\Pi$ is a pseudosurface created by an identification map $q : \Si \to \Pi$.
Suppose that $x \in \Pi$ has pinchdegree $m$ and $M$ is an open or closed $m$-disk neighborhood of $x$ (possibly $x$ is a regular point and $m=1$).  Then $\ti{M} = q\iv(M)$ is a disjoint union of $m$ open or closed disks in $\Si$, which we call open or closed \emph{sheets over $x$}, respectively.  Each sheet contains a unique parent $p$ of $x$, and we denote this sheet by $\ti{M}_p$.  The set $q(\ti{M}_p)$ is an open or closed disk in $\Pi$, which we call an open or closed \emph{plate at $x$}, respectively, and denote by $M_p$.
(Our notation here is similar to that typically used for covering spaces, although the map $q: \Si \to \Pi$ is not a covering projection unless $q$ is the identity.)
If we remove $x$ from a plate, or the appropriate parent of $x$ from a sheet, we talk about a \emph{punctured} plate or sheet, respectively.
Note that in the topology of $\Pi$, an open plate at a regular point is an open set, but open plates at pinchpoints are not open sets.  Closed plates are always closed sets.

\subsection{Lifting curves}\label{ss:liftcurve}
In this subsection we consider curves in pseudosurfaces.  We first examine what happens when a curve passes through a particular point, which may be a pinchpoint.  We then use this to characterize the situations where a curve can be lifted to a unique curve in the parent surface.

Suppose $\Pi$ is a pseudosurface created by $q : \Si \to \Pi$.
A curve $\ga: J \to \Pi$ is \emph{pinchpoint-finite} if there are only finitely many $t \in J$ for which $\ga(t)$ is a pinchpoint.
Suppose that $\ga$ is pinchpoint-finite, and $\ga$ passes through $x_0 = \ga(t_0)$, where $J$ contains points less than $t_0$.  Let $M$ be a multidisk neighborhood of $x_0$.  Suppose first that $x_0$ is a pinchpoint.
Since $\ga$ is pinchpoint-finite and continuous, there exists $t' < t_0$ such that $\ga([t', t_0))$ contains no pinchpoint and $\ga([t', t_0]) \sb M$.
Now $\ga([t',t_0))$ is a connected subset of $M-\{x_0\}$, so it is a subset of one of the components of $M-\{x_0\}$, which are punctured plates $M_p-\{x_0\}$ for parents $p$ of $x_0$.  In particular, $\ga[t',t_0) \sb M_{p_1} - \{x_0\}$ for some $p_1 \in q\iv(x_0)$, and so $\ga([t',t_0]) \sb M_{p_1}$.  We say that \emph{$\ga$ enters $x_0$ at $t_0$ through $M_{p_1}$}.
It is not hard to show that $p_1$ is independent of the choice of $M$.
If $x_0$ is a regular point, then $M$ is a disk and there is $t' < t_0$ such that $\ga([t,',t_0]) \subseteq M$.
We will still say that \emph{$\ga$ enters $x_0$ through $M_{p_1}$}, where $p_1 = x_0$ and $M_{p_1} = M$.

If we restrict $\ga$ to $\{t \in J \;|\; \ga(t)$ is not a pinchpoint$\}$, and consider the codomain as $\Si$ rather than $\Pi$, the resulting map $\ti\ga$ (a `partial lift' of $\ga$) describes a finite collection of curves in $\Si$, and we can take limits of $\ti\ga$ as in Subsection \ref{ss:curve}.
If we take an arbitrary open neighborhood $U$ of $p_1$ in $\Si$, we can choose the $M$ above so that $\ti M_{p_1} \sb U$.  Then for $t \in [t', t_0)$ we have $\ti\ga(t) = \ga(t) \in M_{p_1}-\{x_0\} = \ti{M}_{p_1}-\{p_1\} \sb U$.
This shows that in fact $p_1 = \lim_{t \to t_0^-} \ti\ga(t)$.

Similarly, if $J$ contains points greater than $t_0$,  there is a plate $M_{p_2}$ through which \emph{$\ga$ leaves $x_0$ at $t_0$}, where $p_2 = \lim_{t \to t_0^+} \ti\ga(t)$.

If $t_0 \in J\ir$ and $\ga(t_0)=x_0$ we say that $\ga$ \emph{crosses $x_0$ from plate $M_{p_1}$ to plate $M_{p_2}$ at $t_0$} if $\ga$ enters and leaves $x_0$ at $t_0$ using different plates $M_{p_1}$ and $M_{p_2}$, respectively.  In other words, $\lim_{t \to t_0^-}  \ti\ga(t) = p_1 \ne p_2 = \lim_{t \to t_0^+} \ti\ga(t)$.
We also say that $\ga$ \emph{internally crosses $x_0$}.  Otherwise $\ga$ \emph{skims $x_0$ at $t_0$} (this is always true if $x_0$ is a regular point).
If $\ga:[a,b] \to \Pi$ is a closed curve with $x_0 = \ga(a) = \ga(b)$, then we can say that $\ga$ \emph{terminally crosses $x_0$} (specifying plates if desired) if  $\lim_{t \to b^-} \ti\ga(t) \ne \lim_{t \to a^+} \ti\ga(t)$; otherwise $\ga$ \emph{terminally skims $x_0$}.
Crossings only occur at pinchpoints.

\begin{lemma}[Curve lifting]\label{curvelift}
Let $\Pi$ be a pseudosurface created by $q : \Si \to \Pi$, and let $\ga:[a,b] \to \Pi$ be a pinchpoint-finite curve.
Then there is a unique \emph{lifting} or \emph{parent curve} $\ga_0 : [a,b] \to \Si$, a curve in $\Si$ such that $\ga = q \circ \ga_0$, if and only if $\ga$ does not internally cross a pinchpoint.
Moreover, if $\ga_0$ exists and $\ga$ is closed, then $\ga_0$ is closed if and only if $\ga$ does not terminally cross a pinchpoint.
\end{lemma}

\begin{proof}
Suppose $\ga$ does not internally cross a pinchpoint.
Consider the map $\ti\ga$ from a subset of $[a,b]$ to $\Si$, as discussed above.  Define $\ga_0 : [a,b] \to \Si$ as follows.
If $\ga(t)$ is not a pinchpoint, define $\ga_0(t) = \ga(t) = \ti\ga(t)$.
Now suppose $\ga(t)$ is a pinchpoint.
If $t \in (a,b)$, define $\ga_0(t) = \lim_{u \to t^-} \ti\ga(u)$, which is also equal to $\lim_{u \to t^+} \ti\ga(u)$ because $\ga$ does not internally cross pinchpoints.
If $t=a$, let $\ga_0(a) = \lim_{u \to a^+} \ti\ga(u)$, and if $t=b$, let $\ga_0(b) = \lim_{u \to b^-} \ti\ga(u)$.
We have filled in the missing values for $\ti\ga$ to obtain $\ga_0$ using limits, which give the unique values that make $\ga_0$ continuous.
The condition for $\ga_0$ being closed follows from the definitions of $\ga_0(a)$ and $\ga_0(b)$.

If $\ga$ internally crosses a pinchpoint $\ga(t)$ then we cannot find $\ga_0$, because there is no single value of $\ga_0(t)$ that will make $\ga_0$ continuous at $t$.
\end{proof}

 We will handle curves that internally cross a pinchpoint by breaking them into segments.  Since a simple closed curve $\ga$ in a pseudosurface that does not cross any pinchpoints has a lifting $\ga_0$ that is a simple closed curve, we can define $\ga$ to be $1$-sided or $2$-sided according to the number of sides of $\ga_0$. 

\subsection{Crushing curves}\label{ss:contractcurve}
As mentioned above, we will need to identify the image of a curve $\ga$ in a pseudosurface $\Pi$ to a single point, which we refer to as \emph{crushing $\ga$}.
In particular, $\ga$ will represent an edge in an embedded graph, and so it will be either a simple path or a simple closed curve.
Here we discuss the effect on a surface or pseudosurface of crushing such a curve and show that the result is always a pseudosurface.

\begin{lemma}\label{surfcrush}
Crushing a simple path in a surface results in a surface homeomorphic to the original surface.  Crushing a $1$-sided simple closed curve in a surface results in a surface not homeomorphic to the original surface.  Crushing a $2$-sided simple closed curve in a surface results in a pseudosurface with exactly one pinchpoint, of pinchdegree $2$.
\end{lemma}

\begin{proof}
We omit some technicalities involving regular neighborhoods and provide the main idea of the proof.
Suppose $\Si$ is a surface, $\ga$ is a curve in $\Si$, $C$ is the image of $\ga$, and $r$ is an identification map crushing $\ga$, i.e., identifying $C$ to a point $y$ in $\Pi/r$.
Let $\Si'$ be the component of $\Si$ containing $C$; other components of $\Si$ are unaffected by $r$.
Since $C$ is closed, and we are quotienting it to a single merged output $y$, by Observation \ref{cifo} it suffices to find a suitable open neighborhood base for $y$ consisting of $m$-disks for some $m$.

If $\ga$ is a simple path then there are open disks $B$ in $\Si'$ with $C \sb B$, which become open disk neighborhoods of $y$.  Hence, $y$ is a regular point, and $\Si'/r$ is a surface homeomorphic to $\Si$.

If $\ga$ is $1$-sided then when we crush $C$ to $y$ each open \mobius. strip $B$ in $\Si'$ containing $C$ becomes an open disk, so $y$ is a regular point.
In this case $\Si'/r$ is a surface with Euler genus one smaller than $\Si'$ (loosely, we have removed a crosscap; the result may be orientable or nonorientable).

If $\ga$ is 2-sided then each open cylinder $B$ in $\Si'$ containing $C$ becomes an open $2$-disk with center $y$: the two plates at $y$ come from the two smaller cylinders into which $C$ cuts $B$.  Thus, $y$ is a pinchpoint of pinchdegree $2$, the only pinchpoint in $\Si'/r$, and $\Si'/r$ is a pseudosurface.
Let $\Si''$ be the parent surface of $\Si'/r$ obtained by dismantling $y$.
If $\ga$ is a separating curve in $\Si'$ ($\Si'-C$ is disconnected) then $\Si''$ has two components.  Otherwise, $\Si''$ has one component, with Euler genus $2$ smaller than $\Si'$ (loosely we have removed a handle, or a twisted handle formed by two crosscaps; the result may be orientable or nonorientable).

\end{proof}

\begin{corollary}\label{surfmulticrush}
Suppose we have several simple paths in a surface with disjoint images $C_1, C_2, \dots, C_k$.  If we identify each $C_i$ to a new point $c_i$ for $i = 1, 2, \dots, k$, then the result is a surface homeomorphic to the original surface.
\end{corollary}

\begin{proof}
By Lemma \ref{qp}, the specified identification map is equivalent to crushing the simple paths one at a time.  After each crushing we still have a surface homeomorphic to the original surface, by Lemma \ref{surfcrush}.  The remaining simple paths are still simple paths and still disjoint, so we can repeat, until all identifications have been made.
\end{proof}

Now we can state a result which will be important for defining minor operations for embeddings of graphs in pseudosurfaces.

\begin{theorem}\label{curvecontract}
Crushing a simple path or simple closed curve in a pseudosurface results in a pseudosurface.
\end{theorem}

\begin{proof}
Suppose that $\Pi$ is a pseudosurface created by $q_1 : \Si \to \Pi$. Let $P$ be the set of merged outputs of $q_1$ (the pinchpoints of $\Pi$) and $\ti P = q_1\iv(P)$ the set of merging inputs.  Both $P$ and $\ti P$ are finite.
Suppose $\ga: [a,b] \to \Pi$ is a simple path or simple closed curve in $\Pi$ with image $C$, and let $\ti C = q_1\iv(C)$.
Note that $\ga$ is pinchpoint-finite because $P$ is finite and each $p \in P$ is $\ga(t)$ for at most one $t \in [a,b]$, unless $\ga$ is closed and $p=\ga(a)=\ga(b)$.
Let $q_2$ be the identification map corresponding to crushing $C$ in $\Pi$ to a point $c$.  We wish to show that $\Pi/q_2 = \Si/(q_2 \circ q_1)$ is a pseudosurface.  The map $q_2 \circ q_1$ on $\Si$ sends all elements of $\ti C$ to $c$, and each $s \in \ti P - \ti C$ to $q_1(s)$, while preserving all other points.

First suppose that $\ga$ is a simple path, and label the pinchpoints crossed internally by $\ga$ as $\ga(a_1)$, $\ga(a_2)$, $\dots$, $\ga(a_{k-1})$ where $k \ge 1$ and $a=a_0 < a_1 < a_2 < \dots < a_{k-1} < a_k=b$.
For each $i$ with $1 \le i \le k$, the restriction of $\ga$ to $[a_{i-1}, a_i]$ does not cross any pinchpoints internally, so by Lemma \ref{curvelift}, it has a lifted curve $\ga_i : [a_{i-1}, a_i] \to \Si$, which is a simple path because $\ga$ was a simple path.  Let the image of $\ga_i$ be $C_i \sb \ti C$, and let $C_0 = C_1 \cup C_2 \cup \dots \cup C_k \sb \ti C$.  All points of $\ti C - C_0$ are preimages of pinchpoints, so $\ti C - C_0$ is a subset of $\ti P \cap \ti C$ and hence is finite.

Because $\ga$ crosses a pinchpoint at each $a_i$ with $1 \le i \le k-1$, we know that $\ga_{i}(a_i) \ne \ga_{i+1}(a_i)$ for each such $i$, and so $C_1, C_2, \dots, C_k$ are disjoint.
Thus, if we crush $C_i$ to a point $c_i$ for all $i$ with $1 \le i \le k$, defining an identification map $r_1$, then $\Si'=\Si/r_1$ is a surface by Corollary \ref{surfmulticrush}, which still contains the sets $\ti P - \ti C$ and $\ti C - C_0$.
Now we define a finite identification map $r_2$ on $\Si'$ by sending all elements of the finite set $\{c_1, c_2, \dots, c_k\} \cup (\ti C - C_0)$ to $c$, and all $s \in \ti P - \ti C$ to $q_1(s)$.
It is not difficult to check that $r_2 \circ r_1 = q_2 \circ q_1$.
Therefore, $\Pi/q_2 = \Si/(q_2 \circ q_1) = \Si/(r_2 \circ r_1) = \Si'/r_2$, which is a pseudosurface because $\Si'$ is a surface and $r_2$ is a finite identification map.

If $\ga$ is a simple closed curve that crosses at least one pinchpoint (either internally or terminally), we first cyclically reparameterize $\ga$ so that $\ga$ terminally crosses a pinchpoint, which does not change its image.  Then we can apply almost exactly the same argument as when $\ga$ is a simple path, splitting $\ga$ at internally crossed pinchpoints.
The only difference is that in showing that $C_1, C_2, \dots, C_k$ are disjoint we also use the fact that $\ga_1(a_0) \ne \ga_k(a_k)$ because $\ga$ terminally crosses $\ga(a_0)=\ga(a_k)$.

What remains is when $\ga$ is a simple closed curve that does not cross any pinchpoints.  Then by Lemma \ref{curvelift} we can lift $\ga$ to a simple closed curve $\ga_0$ in $\Si$, with image $C_0$.  
(If $\ga$ skims pinchpoints then $C_0$ contains merging inputs, but this will not cause any problems.)
As before, $\ti C - C_0 \sb \ti P \cap \ti C$ is finite.
Let $r_1$ be the identification map on $\Si$ that crushes $C_0$ to a point $c_0$. Then $\Om = \Si/r_1$ still contains the sets $\ti P - \ti C$ and $\ti C - C_0$.
Now we define a finite identification map $r_2$ on $\Om$ by sending all elements of the finite set $\{c_0\} \cup (\ti C - C_0)$ to $c$, and all $s \in \ti P - \ti C$ to $q_1(s)$.
It is not difficult to check that $r_2 \circ r_1 = q_2 \circ q_1$. Therefore, $\Pi/q_2 = \Si/(q_2 \circ q_1) = \Si/(r_2 \circ r_1) = \Om/r_2$.
By Lemma \ref{surfcrush} we know that if $\ga$ is $1$-sided, i.e., $\ga_0$ is $1$-sided, then $\Om$ is a surface, and hence $\Pi/q_2 = \Om/r_2$ is a pseudosurface by definition.
If $\ga$ is $2$-sided, i.e., $\ga_0$ is $2$-sided, then $\Om$ is a pseudosurface, and so $\Pi/q_2 = \Om/r_2$ is a pseudosurface by Observation \ref{quotps}.
\end{proof}

From the above proof we see that the parent surface (up to homeomorphism) changes only in the case where $\ga$ is a simple closed curve that does not cross any pinchpoints (internally or terminally). 
In the other cases, $\Si'$ is the parent surface of the new pseudosurface, but $\Si'$ is homeomorphic to the original parent surface $\Si$.

Theorem \ref{curvecontract} and its proof do not give details of the structure of the new pseudosurface.  These details can be worked out, but this is not necessary for our purposes.  We mention a heuristic that extends \cite[p.~99]{EM15}: when the image $C$ of a simple path or simple closed curve in a pseudosurface $\Pi$ is crushed to a point $c$, a neighborhood of $c$ in the new pseudosurface can be determined by `fattening' $C$ to include a small neighborhood of each of its points in $\Pi$; each boundary component of the fattened image corresponds to a plate at $c$.

Theorem \ref{curvecontract} will be used frequently and without explicit reference in the rest of this paper.

\section{Graphs and embeddings}\label{sec:graphemb}

Our framework for embeddings of graphs in pseudosurfaces builds on the theory of embeddings in surfaces.  We therefore need to establish basic definitions for graphs and embeddings of graphs in surfaces, and briefly summarize duality, twisted duality, and minors for cellular embeddings of graphs in surfaces, which we do in this section.
Readers familiar with basic ideas in topological graph theory may be able to skip or skim most of this section, with the exception of Subsection \ref{ss:minorsurf}, which raises issues regarding the definition of minor operations that are important later.

\subsection{Graphs}\label{ss:graph}
In this paper graphs are finite, and may have loops and multiple edges.  It is therefore most convenient to regard a graph as consisting of a finite set of vertices and a finite set of half-edges.  The half-edges are paired into edges, and each half-edge is incident with a unique vertex.  Other than regarding half-edges, rather than edges, as atomic (i.e., fundamental) objects, we follow standard graph theory terminology and notation as provided by West \cite{West}.

When convenient we allow slightly extended versions of graphs.  We sometimes allow graphs to have \emph{semiedges}, which are half-edges, incident with a unique vertex, that do not form part of a full edge.
Semiedges will be useful in analyzing embeddings with `pinched edges'.
We sometimes also allow graphs to have \emph{free loops}, which are atomic objects that are considered to be cycles of length $0$.  Specifically, free loops will be allowed in gems and jewels (see Subsection \ref{ss:repembsurf}) but not elsewhere.

Graphs have minor operations, consisting of deleting and contracting edges, and deleting vertices.
An edge of a graph is a \emph{loop} if its two incident vertices are the same; otherwise it is a \emph{link}.
When we contract an edge we delete the edge and identify its ends.  If we contract a link $e$, any edges parallel to $e$ become loops.
Contraction of a loop is the same as deletion.
Deletion of a vertex means deletion of all incident edges followed by deletion of the (now isolated) vertex itself.
We can also delete and contract semiedges and free loops: deletion is defined in the obvious way, and contraction is defined to be the same as deletion.

When discussing embeddings of graphs we want to also regard our graphs as topological spaces.  We therefore think of each vertex as a single point, and each edge as a copy of the interval $[0,1]$, and perform appropriate identifications to join the ends of each edge to the appropriate vertex or vertices.  The graph is then represented by a $1$-dimensional CW complex (1-complex).
If we allow semiedges, each semiedge is represented by a copy of the interval $[0,\frac12]$ with the point corresponding to $0$ identified with a vertex.  If we allow free loops, each free loop corresponds to a copy of a circle $S^1$, disjoint from the rest of the graph.
When no confusion will result, we will not distinguish between a graph and its $1$-complex.

An embedding of a graph in a topological space, or \emph{embedded graph}, is a triple $G=(\Ga, X, \phi)$ where $\Ga$ is the \emph{underlying graph}, $X$ is the \emph{underlying topological space}, and $\phi$ is a homeomorphism from $\Ga$ (considered as a $1$-complex) to a subspace of $X$.  Often we use $G$ to mean the image $\phi(\Ga) \sb X$, so we talk about `an embedded graph $G \sb X$', and we use $V(G)$ and $E(G)$ to mean the images of the vertex and edge sets of $\Ga$, respectively.

When we refer to the \emph{interior} of an edge $e$ of a graph, denoted by $e\ir$, we mean the edge with its incident vertices deleted, which is homeomorphic to the open interval $(0,1)$.  We use this terminology and notation even when $e$ is embedded in a larger space $X$, where $e\ir$ is often not actually the interior of $e$ as a subset of $X$.

\subsection{Embeddings of graphs in surfaces}
\label{ss:embgraphsurf}
Suppose $\Si$ is a surface and we have an embedded graph $G \sb \Si$ (as defined in Subsection \ref{ss:graph}).  Then a \emph{face} of $G$ is a component of $\Si-G$, and we use $F(G)$ to denote the set of faces.  A face is \emph{supported} if its closure in $\Si$ intersects the embedded graph $G$, and \emph{unsupported} otherwise.  Unsupported faces are components of $\Si$ that do not intersect the embedded graph.

It is common to consider only \emph{cellular} embeddings, where every face is homeomorphic to an open disk.  The importance of these has been realized for a long time; for example, the generalized version of Euler's formula requires an embedding to be cellular.
Youngs \cite{You63} showed that minimum orientable genus embeddings and minimum Euler genus embeddings of connected graphs must be cellular, and defined the \emph{capping} or \emph{cellularizing} operation by which a general embedding can be reduced to a cellular embedding while preserving many features of the embedding.
This involves removing any faces that are not homeomorphic to a disk, including all unsupported faces, and gluing disks to all resulting boundary components.  Some care needs to be exercised if faces are removed from both sides of an edge.

Cellular embeddings have a number of other representations in addition to the embedding itself, some of them purely combinatorial.  These representations determine the cellular embedding up to homeomorphism.
Note that in a cellular embedding an isolated vertex must be embedded in a surface component that is a sphere containing no other part of the graph.
Cellular embeddings can include free loops, each as a circle embedded in its own sphere $S_0$.
They can also include semiedges, each as a simple path starting at a vertex and otherwise disjoint from the graph.  The end of an embedded semiedge that is not a vertex will be called the \emph{tip} of the semiedge.

Not all embeddings of graphs in surfaces are cellular.
Moreover, some operations that we want to perform on embedded graphs do not always preserve cellularity of an embedding.  These include deleting an edge and twisting (taking a partial Petrie dual with respect to) an edge.  To stay in the class of cellular embeddings, we often perform some operation and then cellularize the result.

\subsection{Representations of cellular embeddings in surfaces}
\label{ss:repembsurf}
We assume the reader is generally familiar with the following common representations of cellular embeddings, but provide brief descriptions and references.

A \emph{signed rotation system} (see, e.g., \cite[Subsection 3.2.3]{GT} or \cite[Section 3.3]{MT}) is a pair $(\rho, \sigma)$ where $\rho$ is a permutation of the half-edges of $G$ and $\sigma$ assigns a sign to each edge.  The cycles of $\rho$ describe the cyclic order of half-edges around each vertex relative to a choice of a local positive (clockwise) direction at each vertex.  For a given edge $e$, $\sigma(e)$ is either $+1$ or $-1$ according to whether the local clockwise direction is preserved or reversed, respectively, as $e$ is traversed.
If present, semiedges are included in the rotation at their incident vertex but do not have a sign.

A \emph{band decomposition} (see, e.g., \cite[Subsection 3.2.1]{GT}) is obtained by `fattening' the vertices of $G$ in $\Si$ to become closed disks, and the edges of $G$ to become rectangles (topologically also closed disks), so that each face of $G$ is also represented by a closed disk.  The closed disks representing vertices, edges, and faces are known as \emph{$0$-, $1$- and $2$-bands}, respectively.
If $\{i,j,k\} = \{0,1,2\}$, the boundary of each $i$-band is divided into an even number of segments, with segments shared with a $j$-band alternating with segments shared with a $k$-band.  For $1$-bands, the number of segments is exactly $4$.
An isolated vertex is represented by a surface component that is a sphere, divided by a circle into a $0$-band and a $2$-band.
A semiedge can be represented by a $1$-band whose boundary can be divided into two segments, one shared with a $0$-band and the other shared with a $2$-band.

A \emph{ribbon graph} (see, e.g., \cite{BR02, EM13}) is a surface with boundary obtained from a band decomposition by deleting the interiors of the $2$-bands. It therefore contains closed disks representing vertices, and closed disks (often called \emph{ribbons}) representing edges.  The boundary components (abbreviated to \emph{boundaries}) represent the faces.  Ribbon graphs  can also be defined as collections of vertex disks and edge disks satisfying certain intersection properties, and are also known as \emph{fatgraphs}.
Isolated vertices are represented by an isolated vertex disk.
Ribbon graphs incorporate semiedges as disks that share exactly one boundary segment with a vertex disk.

Gems and jewels are purely combinatorial representations using edge-colored graphs (see, e.g., \cite{BL95, Lins82}).
A \emph{gem} $K$ is a properly $3$-edge-colored cubic graph that may be regarded as the $1$-skeleton of a band decomposition $B$ (or ribbon graph), where the edges of $K$ are colored according to the types of band of $B$ on either side.  We color an edge of $K$ with color $\cv$, $\ca$, or $\cf$, according to whether the bands on each side of $e$ are $0$- and $1$-bands, $0$- and $2$-bands, or $1$- and $2$-bands, respectively.
The edges of any two colors $c_1, c_2$ then form a $2$-factor, and we refer to its component cycles as $c_1c_2$-cycles.  The colors $c_1$ and $c_2$ alternate along each $c_1c_2$-cycle.  The $\ca\cv$-cycles represent vertices, so we call them \emph{v-gons}, and the $\ca\cf$-cycles represent faces, so we call them \emph{f-gons}.  The $\cv\cf$-cycles all have length $4$, and represent edges, so we call these \emph{edge-squares} or just \emph{e-squares}.

A \emph{jewel} is a properly $4$-edge-colored $4$-regular graph obtained from a gem by adding two diagonal edges colored $\cz$ to each edge-square (we will also refer to the resulting copy of $K_4$ as an edge-square when no confusion will result).  The $\ca\cz$-cycles in a jewel represent \emph{Petrie walks}, also known as \emph{zigzags}, in the embedding, so we call these cycles \emph{z-gons}.
To represent isolated vertices embedded in a spherical surface component, gems and jewels may contain free loops, which are considered to have color $\ca$, so that each free loop is a v-gon, an f-gon, and (in a jewel) also a z-gon.
To incorporate semiedges, a gem may contain $\cv\cf$-cycles of length $2$, i.e., parallel edges of colors $\cv$ and $\cf$, which we call \emph{semiedge-digons}.  In a jewel we add a third parallel edge of color $\cz$ to each semiedge-digon.

One operation we use for graphs derived from gems or jewels is \emph{smoothing along $c$}, where usually $c \in \{\cv, \cf, \cz\}$.  Typically we will choose one or more e-squares or semiedge-digons and (1) remove from each all edges except those of color $c$, and change the edges of color $c$ to color $\ca$.  The vertices of the e-square or semiedge-digon now have degree $2$, being incident with two edges of color $\ca$.  We now (2) replace each maximal path of edges of color $\ca$ by a single edge of color $\ca$ (deleting all internal vertices), and each cycle of edges of color $\ca$ by a free loop of color $\ca$ (deleting all vertices).
The result is a new jewel or gem.
Smoothing will be used to define edge minor operations for gems and jewels.

We will use gems and jewels as our primary combinatorial representation for graphs embedded in surfaces and also (with some extensions) for graphs embedded in pseudosurfaces.

\subsection{Duality for embeddings in surfaces}\label{ss:dualsurf}

To form a dual $G\du$ for an embedding of a graph $G$ in a surface, we put one dual vertex in each face, and add a dual edge $e^*$ crossing a primal edge $e$ and joining the dual vertices in the two faces (which may be the same face) on either side of $e$.
This is sometimes called the \emph{geometric dual} to distinguish it from other duals (discussed below).
Each semiedge $d$ has a dual semiedge $d^*$, obtained by joining the dual vertex in the unique face to which $d$ is incident to the tip of $d$.
See Figure~\ref{fig:dualsemiedge} for an example of a planar embedding of a graph with semiedges: the primal graph is black, the dual is red, and tips of semiedges are indicated by small open circles.
If the primal embedding is cellular, the dual embedding is cellular, and unique (up to homeomorphism).

\begin{figure}[ht]
\hbox to\hsize{%
  \hss%
\begin{tikzpicture}[scale=0.85]
 \begin{scope}[rotate=-90]
  \coordinate (A)  at ( 0.0,7.30);
  \coordinate (UL) at (-1.8,5.85);
  \coordinate (UR) at ( 1.8,5.85);
  \coordinate (ML) at (-1.8,3.30);
  \coordinate (MR) at ( 1.8,3.30);
  \coordinate (LL) at (-1.8,1.45);
  \coordinate (LR) at ( 1.8,1.45);
  \coordinate (R1) at ( 0.0,6.30);
  \coordinate (R2) at ( 0.0,4.40);
  \coordinate (R3) at (-1.0,3.90);
  \coordinate (J)  at ( 0.0,0.00);

  \draw[red] (R1) .. controls (1.04,6.90) and (3.40,6.20) ..
    (3.40,2.50) .. controls (3.40,0.20) and (2.00,-0.70) .. (J);
  \draw[red] (J) .. controls (-2.00,-0.70) and (-3.40,0.20) .. (-3.40,2.50)
    .. controls (-3.40,6.20) and (-1.04,6.90) .. (R1);
  \draw[red] plot[smooth,tension=0.75] coordinates
     {(J) (-1.45,0.40) (-2.65,2.40) (-2.10,4.30) (R2)};   
  \draw[red] plot[smooth,tension=0.75] coordinates
     {(R2) (2.10,4.30) (2.65,2.40) (1.45,0.40) (J)};
  \draw[red] (R1) -- (J);
  \draw[red] (LL) .. controls (-1.05,0.65) .. (J);
  \draw[red] (LR) .. controls ( 1.05,0.65) .. (J);
  \draw[red] (R3) -- (R2);

  \draw[blk] (A) -- (UL) -- (UR) -- (A);
  \draw[blk] (ML) -- (MR);
  \draw[blk] (UL) -- (LL);
  \draw[blk] (UR) -- (LR);
  \draw[blk] (R3) -- (ML);        

  \node[bv] at (A) {}; \node[bv] at (UL) {}; \node[bv] at (UR) {};
  \node[bv] at (ML) {}; \node[bv] at (MR) {};
  \node[tv] at (LL) {}; \node[tv] at (LR) {};
  \node[rv] at (R1) {}; \node[rv] at (R2) {}; \node[tv] at (R3) {};
  \node[rvbig] at (J) {};

 \end{scope} 
\end{tikzpicture}
  \hss%
}
\caption{Dual of an embedding with semiedges}\label{fig:dualsemiedge}
\end{figure}
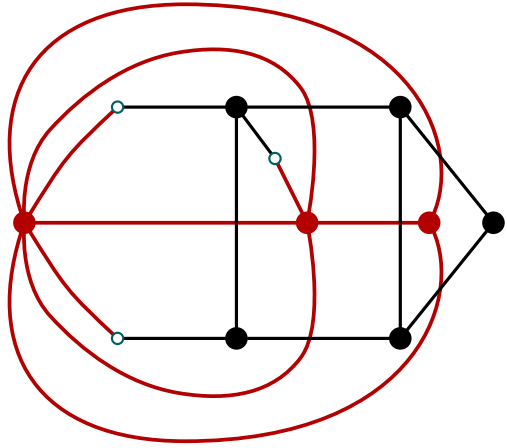

To take the dual of a band decomposition representing a cellular embedding, we just relabel $0$-bands as $2$-bands and vice versa.  The $1$-bands representing edges (or semiedges) are unchanged (so each primal edge $e$ and its dual edge $e^*$ are represented by the same $1$-band).  So in some representations (in particular, band decompositions, ribbon graphs, gems and jewels) the primal edge $e$ and the dual edge $e^*$ are essentially the same object, and we refer to both as $e$, but for actual embeddings in surfaces we distinguish between them.

To take the dual of a ribbon graph, we glue a disk to each boundary (representing a face) and then delete the interiors of the vertex disks.  The new disks represent the vertices of the dual embedding.  The edge disks in the primal ribbon graph remain as edge disks in the dual ribbon graph.

To take the dual of a gem (or jewel), we just exchange the colors $\cv$ and $\cf$ in the coloring of the edges.  Thus, the v-gons in the primal gem, representing vertices, become f-gons in the dual gem, representing faces, and vice versa.

We will use the following classification for edges of a graph embedded in a surface.
Recall that an edge of a graph is either a loop or a link.
We say a loop is \emph{twisted} or \emph{untwisted} according to whether it is $1$-sided or $2$-sided, respectively, as a closed curve in the surface.
The dual edge of a link is a \emph{colink}, and the dual edge of a twisted or untwisted loop is a \emph{twisted} or \emph{untwisted coloop}, respectively.
Alternatively, we may characterize a colink as an edge with distinct faces on its two sides, and a coloop has the same face on both sides.
If we assign a direction to the boundary of the face on both sides of a coloop, the coloop is twisted or untwisted according to whether the boundary passes through it twice in the same direction, or once in each direction, respectively.

There are thus nine possibilities for an edge: it can be one of a link, twisted loop, or untwisted loop, and also one of a colink, twisted coloop, or untwisted coloop.  All nine possibilities do in fact occur.

\subsection{Twisted duality for embeddings in surfaces}
\label{ss:twdualsurf}

Cellular embeddings of graphs in surfaces actually have a rich theory involving three duality operations, and partial versions of these operations (applied to individual edges or subsets of edges).
The \emph{dual} of an edge-based operation $o_1$ is $o_2$ defined by $G \,o_2\, e = (G\du \,o_1\, e\du)\du$ for an edge $e$, or $G \,o_2\, A = (G\du \,o_1\, A\du)\du$ if $o_1$ can unambiguously be applied to a set $A$ of edges.

First, in addition to the geometric dual described above, we can also form the \emph{Petrie dual $G\pe$} of a cellularly embedded graph $G$.  This amounts to twisting all edges, i.e., changing the sign of all edges in a signed rotation system representation.  It corresponds to giving all edge ribbons a half-twist in a ribbon graph, or swapping the colors $\cf$ and $\cz$ in a jewel (which can be translated into an action on a gem which modifies the $4$-cycle in an e-square).

As pointed out by Wilson \cite{Wil79}, the geometric and Petrie dual operations are involutions that satisfy the relation $G\du\pe\du = G\pe\du\pe$.
Thus, applying the geometric and Petrie dual operations to an embedded graph $G$ generates at most six distinct embeddings $G$, $G\du$, $G\pe$, $G\du\pe$, $G\pe\du$ and $G\du\pe\du = G\pe\du\pe$.
Considering the geometric and Petrie duals as operations on the set of all embedded graphs, they generate an action of the symmetric group $S_3$.
In terms of jewels, this group permutes the three edge colors $\cv$, $\cf$ and $\cz$, leaving $\ca$ fixed.  The embedding $G\du\pe\du = G\pe\du\pe$ is often called the \emph{Wilson dual $G\wi$}.  The Wilson dual operation is an involution, and in a jewel it corresponds to swapping the colors $\cv$ and $\cz$.

The geometric, Petrie, and Wilson duals can also be applied to subsets of the edges of an embedded graph to give \emph{partial duality} operations.  Partial Petrie duals twist individual edges, or a subset of edges; these operations have been used for a long time.  In terms of an actual embedding, the partial Petrie dual with respect to a set $A$ of edges amounts to inserting a crosscap in the surface in the middle of each edge in $A$, and then cellularizing the resulting embedding.
However, partial geometric duality, usually referred to as just partial duality, was not defined until 2009, by Chmutov \cite{Chm09}.
Partial geometric duality and partial Petrie duality generate an action of the group $S_3^E$ on cellularly embedded graphs with edge set $E$, as described by one of us (J.E.-M.) and Moffatt \cite{EM12}; the resulting operations are generally known as \emph{twisted duals}.  In terms of jewels, this amounts to specifying an individual permutation of the colors $\cv, \cf, \cz$ for each e-square.
We use $G \pdu A$ and $G \ppe A$ to represent the partial dual and partial Petrie dual of $G$ with respect to a set of edges $A$, and use notation such as $G \pdu A \ppe B$ for $(G \pdu A) \ppe B$ and $G \pdu\ppe A$ for $G \pdu A \ppe A$.  In this setting we do not distinguish between $e$ and $e\du$ or $A$ and $A\du$.  A partial Wilson dual can be defined as $G \pwi A = G \pdu\ppe\pdu A$, and we also have $G\pwi A = G \ppe\pdu\ppe A = (G\du \ppe A)\du$, so partial Wilson dual is the dual operation for partial Petrie dual.

\begingroup
\def\po{1\to2}\def\mo{2\to1}\def\oo{1\to1}

The effects of the three different partial duals on vertices and faces will be important for our analysis.
We consider partial duals with respect to an individual edge $e$ in a cellularly embedded graph $G$.  Only vertices and faces that are incident with $e$ are affected.
The results are summarized in Table \ref{tab:effpd}.  Here $\po$ indicates that a single object (vertex or face) is split into two; $\mo$ indicates that two objects are combined into one; and $\oo$ indicates that there is a single object that remains a single object, although the ordering of the edges (rotation at a vertex, or boundary walk of a face) associated with the object may change.
(Note that partial Petrie duals and partial Wilson duals affect the zigzags incident with $e$; partial geometric duals do not.  Showing the effect on zigzags would make the table appear more symmetric.)

\begin{table}[ht]
 \begin{tabular}{|l|c|c|c|}
  \hline
  \multicolumn{4}{|l|}{$G \mapsto G \ppe e$ affects faces but not vertices} \\
  \hline
  Status of $e$ in $G$ & colink & untw.~coloop & tw.~coloop \\
  Effect on faces incident with $e$ & $\mo$ & $\oo$ & $\po$ \\
  \hline
 \noalign{\vskip10pt}
  \hline
  \multicolumn{4}{|l|}{$G \mapsto G \pwi e$ affects vertices but not faces} \\
  \hline
  Status of $e$ in $G$ & link & untw.~loop & tw.~loop \\
  Effect on vertices incident with $e$ & $\mo$ & $\oo$ & $\po$ \\
  \hline
 \noalign{\vskip10pt}
  \hline
  \multicolumn{4}{|l|}{$G \mapsto G \pdu e$ affects vertices and faces} \\
  \hline
  Status of $e$ in $G$ & link & untw.~loop & tw.~loop \\
  Effect on vertices incident with $e$& $\mo$ & $\po$ & $\oo$ \\
  \hline
  Status of $e$ in $G$ & colink & untw.~coloop & tw.~coloop \\
  Effect on faces incident with $e$ & $\mo$ & $\po$ & $\oo$ \\
  \hline
 \end{tabular}
 \bigskip
 \caption{Effects of different partial dualities}\label{tab:effpd}
 \end{table}

\endgroup

\subsection{Minors for embeddings in surfaces}\label{ss:minorsurf}
Now we discuss edge-based minors for graphs cellularly embedded in surfaces.
Minor operations for abstract graphs form the basis of the famous Robertson-Seymour graph minors project.
For graphs we can take minors in the class of general graphs (loops and multiple edges allowed), or restrict to the class of simple graphs.  We know that, for example, embeddability in a given surface can be characterized by a finite number of forbidden minors, which are simple graphs.
Minor operations are also important for embedded graphs, but for a long time it was unclear how these should be defined in general.  There are seemingly natural ways to define deletion or contraction of an edge, which correspond to deletion and contraction of edges in abstract graphs, but unfortunately they can both create embeddings that are no longer cellular embeddings in a surface.

To delete an edge, the natural operation on the embedding is \emph{removal}: we remove all interior points of the edge from the image of the embedding.  This works for colinks.  For a twisted coloop, this results in a face homeomorphic to a \mobius. strip, with one boundary component.  For an untwisted coloop, this results in a face homeomorphic to a cylinder, with two boundary components.  In the last two cases the embedding is no longer cellular.

To contract an edge or semiedge, the natural operation on the embedding is \emph{quotienting}: we take the topological quotient where the image of the edge or semiedge is identified to a single point.
Lemma \ref{surfcrush} helps to analyze what happens.
A link (simple path) can be contracted without changing the surface or affecting cellularity of the embedding.  A twisted loop ($1$-sided simple closed curve) can be contracted without affecting cellularity of the embedding, but this does change the surface, effectively removing a crosscap.  Contraction of an untwisted loop ($2$-sided simple closed curve) creates a pinchpoint, so we no longer have a surface.

If we work with loopless graphs, then the problems with edge deletion can be fixed by cellularization, and there are no problems with edge contraction, assuming we delete any created loops.  However, loops are important: for example, even if a graph has no loops, its dual can easily have loops.  So working only with loopless graphs is not a reasonable option.
The issues with edge contraction and deletion were resolved by \bollobas. and Riordan \cite{BR02} by working with ribbon graphs.  The significance of their definitions became clear with the introduction of partial duality by Chmutov \cite{Chm09}, since these definitions interact in very natural ways with partial duality.

Working with ribbon graphs, deletion of an edge $e$ in a ribbon graph $R$ always results in a new ribbon graph $R\dt e$, and the graph embedding corresponding to $R\dt e$ is automatically cellular.  Contraction of edges can then be defined as the dual operation to deletion: $R \ct e = (R^* \dt e)^*$.
For a graph $G$ cellularly embedded in a surface, this means that to form $G\dt e$ by deleting $e$ we use removal for a colink, and for a coloop we use removal and then cellularize the embedding.
Contraction to obtain $G\ct e$ is quotienting for a link or twisted loop, and for an untwisted loop we quotient and then dismantle the resulting pinchpoint, replacing the vertex incident with the loop by two new vertices.  Lemma \ref{curvelift} allows us to uniquely attach the half-edges incident with $v$ (other than the halves of $e$) to one of these new vertices.
We will call these operations the \emph{BR minor} operations.  Later we will contrast these operations with quotient-based operations.  The definitions of BR edge deletion and contraction are simple for ribbon graphs, but more complicated for actual embeddings.  For ribbon graphs, deletion is the natural operation and contraction is defined as the dual of deletion.
BR deletion of an edge means that we delete the edge in the underlying abstract graph.  However, BR contraction of an edge does not correspond to contracting an edge in the underlying abstract graph if the edge is an untwisted loop.

In terms of gems or jewels, the BR operations also work in a simple consistent way (which is to be expected, since gems are closely related to ribbon graphs).
Deleting $e$ corresponds to smoothing the e-square corresponding to $e$ along color $\cv$.
Contracting $e$ is similar, but we smooth along $\cf$.

In a jewel we can perform a third minor operation involving $e$ (see for example \cite[Subsection 4.3.1]{EM13}).  We can smooth along $\cz$ in the e-square corresponding to $e$.
We will call this operation \emph{twist-contraction} because when applied to $e$ we obtain $(G \ppe e) \ct e$, which we abbreviate to $G \tc e$.  It is also sometimes known as \emph{Penrose contraction}.  Since it is a composition of a twist and a contraction, it can also be implemented in other embedding representations.
This operation is self-dual.

We note that for semiedges, the natural operations for deletion (by removal) and contraction (by quotienting) work, and deleting a semiedge is the same as contracting it, or as twist-contracting it.  In a gem or jewel these operations involve smoothing along a color in the semiedge-digon.

Vertex deletion is also considered a minor operation for cellularly embedded graphs.  Deleting a vertex $v$ should be equivalent to first deleting all incident edges and then deleting the resulting isolated vertex, so we need only consider isolated vertices.  In a cellular surface embedding an isolated vertex must be embedded in its own spherical component of the surface, with one incident face.  To delete the vertex, we delete the entire spherical component, which corresponds to deletion of an isolated vertex disk in a ribbon graph, or deletion of a free loop in a gem.

\section{Embeddings in pseudosurfaces and prior work}
\label{sec:embpseudosurf}

In this section we establish goals for a theory of duality and minors for embeddings of graphs in pseudosurfaces, define some necessary concepts, and examine prior work in this area.

\subsection{Goals and obstacles for a theory of duality and minors}
Constructing a duality theory with desirable properties for graphs embedded in surfaces requires restricting the class of embeddings that we are willing to consider.  An arbitrary embedding of a `primal' graph in a surface has a unique dual abstract graph, and moreover the dual graph is embeddable in the same surface.  However, in general the dual embedding is not unique, and the dual embedding of the dual embedding is not necessarily the primal embedding.
However, we would like to have a \emph{true} duality, by which we mean that our set of objects is closed under duality, the dual is unique (up to a natural definition of equivalence), and the dual of the dual is the primal.
To obtain a true dual operation on graph embeddings in surfaces, we need to restrict ourselves to \emph{cellular} embeddings, where each face is homeomorphic to an open disk.

It has been a somewhat vexing problem to try to define a duality theory for pseudosurface embeddings of graphs that has similar properties to the theory for cellular surface embeddings.
If we allow arbitrary embeddings of a graph in a pseudosurface, we cannot even define an abstract dual graph. An edge that goes through several pinchpoints may be part of the boundary of several faces (components of the surface with the embedded graph deleted), not just one or two, so even if we assign a dual vertex to each face, it is not clear which dual vertices a dual edge should join.

The usual solution has been to restrict to embeddings where all pinchpoints correspond to vertices of the graph.  In that case edges do not go through pinchpoints, and all faces are homeomorphic to the interior of a compact surface with boundary.  In this case there is a unique abstract dual graph, which can be embedded in the pseudosurface.
For some purposes such an abstract dual graph is sufficient.  For example, Schluchter and Schroeder \cite{SS17}, also with Rarity \cite{RSS18}, consider self-dual embeddings in pseudosurfaces, where the abstract dual is isomorphic to the original graph.  One of us (J.E.-M.) and Moffatt \cite{EM15} define a Las Vergnas polynomial for graphs embedded in pseudosurfaces using the abstract dual.  For coloring problems, properly coloring the faces of the primal embedding is equivalent to properly coloring the vertices of the abstract dual graph.

If we insist that all faces be open disks, as well as pinchpoints occurring only at vertices, then not only the dual graph, but also its embedding, is unique. However, even in this case the embedding of the dual graph does not belong to our original class of embeddings, where all pinchpoints are vertices of the graph.
So we do not have a true duality theory.

In addition to a true duality theory, we would also like to have a theory that includes the \emph{minor} operations of edge deletion, edge contraction, and (sometimes) vertex deletion.  We would like to be able to apply edge deletion and contraction to all edges, while staying within a well-defined class of objects.  We would also like edge deletion and edge contraction to be dual operations (applying one to the primal corresponds to applying the other to the dual).

 Below we make some formal definitions for embeddings in pseudosurfaces, and then discuss prior efforts to create theories of duality and minors for embeddings of graphs in pseudosurfaces.

\subsection{Definitions for embeddings of graphs in pseudosurfaces}
As with embeddings in surfaces, an embedded graph $G$ in a pseudosurface $\Pi$ is an embedding in the topological sense, i.e., $G \sb \Pi$ is homeomorphic to the $1$-complex representing an abstract underlying graph $\Ga$.
Again, the faces of $G$ are the components of $\Pi-G$ and the set of faces is denoted $F(G)$.

Given a graph embedding $G$ in a pseudosurface $\Pi$, a vertex is \emph{pinched} if it occurs at a pinchpoint, and an edge is \emph{pinched} if its interior contains a pinchpoint.
A vertex $v$ that is not pinched has a neighborhood that is an open disk, so we say $v$ is \emph{cellular}.
If an edge $e$ is not pinched then its interior $e\ir$ has a neighborhood that is an open disk, so we say $e$ is \emph{cellular}.

The situation for faces is more complicated.
As with surface embeddings, a face is \emph{supported} if its closure intersects the embedded graph, and \emph{unsupported} otherwise.
An unsupported face is a component of $\Pi$ that contains no element of $G$.  Unsupported faces are usually of no interest.
A face is \emph{pinched} if it contains a pinchpoint, and \emph{unpinched} otherwise.
An unpinched face is either a connected surface that is a component of $\Pi$, or a face homeomorphic to the interior of a compact connected surface with boundary.
Faces may be any combination of supported or unsupported with pinched or unpinched.
A face is \emph{cellular} if it is homeomorphic to an open disk. Cellular faces are supported and unpinched.

An embedding is \emph{vertex-}, \emph{edge-}, or \emph{face-cellular} if all vertices, edges, or faces, respectively, are cellular.
It is \emph{edge-pinched} if every edge is pinched.
It is \emph{face-supported} or \emph{face-unpinched} if every face is supported or unpinched, respectively.  Face-cellular embeddings are face-supported and face-unpinched.

An embedding of a graph in a pseudosurface is \emph{classic} if it is edge-cellular and face-unpinched, or equivalently if all pinchpoints occur at vertices.  In a classic embedding every unsupported face must be a connected surface.
An embedding is \emph{cellular} if it is edge-cellular and face-cellular, or equivalently classic and face-cellular.
An embedding is cellular and vertex-cellular if and only if it is a cellular embedding in a surface.
 Most prior literature on embeddings of graphs in pseudosurfaces deals with classic embeddings or cellular embeddings.

Sometimes we want to dismantle a pinchpoint of a pseudosurface where there is a vertex of an embedded graph.  In that case each plate of a neighborhood of $v$ receives a new vertex at the location of $v$.  Lemma \ref{curvelift} allows us to attach each half-edge incident with $v$ to a unique new vertex.

We will also use an idea due to Huggett and Moffatt \cite{HM20}, to discuss both their results and other related results. Two pseudosurface embeddings of a given graph are related by \emph{stabilization} if there is a bijection between their sets of supported faces such that if $f_1$ maps to $f_2$ then $f_1$ and $f_2$ have exactly the same collection of boundary walks in the graph.  Given any set of embeddings, this relation forms an equivalence relation on that set, whose equivalence classes are \emph{stabilization classes}.

Our definition of stabilization differs from that of \cite{HM20}, also used in \cite{MT24}.  The definition there is for classic embeddings, and involves a bijection defined on all faces, where $f_1$ can be changed into $f_2$ by adding and removing handles subject to certain restrictions.
However, this has problems: loosely, it does not allow for adding or removing crosscaps, or for adding or removing unsupported faces.
A procedure for constructing a colored ribbon graph from a graph embedded in a pseudosurface is given in \cite[p.~265]{HM20}; two embeddings are supposed to produce the same colored ribbon graph if and only if they are related by stabilization.
Under this procedure the following all have the same colored ribbon graph: (1) a single vertex embedded in a sphere, (2) a single vertex embedded in a projective plane, and (3) a single vertex embedded in one of two spheres. However, adding and removing handles cannot change (1) into (2) because the Euler genera of the surfaces have different parities, and cannot change (1) into (3) because there is no way to add the second sphere.  However, these embeddings are related by stabilization under our definition.
Our definition also applies to embeddings more general than classic embeddings.  It works for embeddings with pinched faces and (if boundary walks are defined appropriately as sequences of half-edges) with semiedges and cross-pinched edges (see Subsection \ref{ss:pseudocellular}).

\subsection{Prior models of duality and minors for pseudosurface embeddings}
\label{ss:prior}
There are some previous works that develop a theory of duality or minors for embeddings of graphs in pseudosurfaces.  We discuss these in some detail, so that we can later explain how they relate to our own framework.
In particular, there are two previous models of true duality for embeddings in pseudosurfaces due to Deneen, Shute, and Thomborson \cite{DST91}.  The first of those was rediscovered in \cite{Dun20}.
The second model of \cite{DST91} is equivalent to models developed independently by Huggett and Moffatt in \cite{HM20}; this corresponds to the quasicellular embeddings we define in Subsection \ref{ss:quasicellular}.

The first model of \cite{DST91} considers classic embeddings in pseudosurfaces where every face is homeomorphic to a \emph{punctured sphere}, a sphere with finitely many holes (finitely many points, or disjoint closed disks, removed).
Each vertex is associated with a set of `origins' which are essentially the rotations in the different plates associated with the vertex.  Faces are called `windows', which may have several `panes' that are their boundary components.
Their duality operation exchanges a vertex with $d$ origins (i.e., of pinchdegree $d$) with a window with $d$ panes (i.e., face that is a punctured sphere with $d$ holes).
They work with oriented pseudosurfaces (although they just refer to them as `orientable') and take care to preserve the global clockwise direction in their duality operation.
This yields a consistent duality theory for embeddings of graphs in oriented pseudosurfaces, but does not extend to nonorientable situations.
Two of us (B.D. and M.E.) \cite[Chapter 6]{Dun20} independently re-invented the construction by Deneen, Shute, and Thomborson that dualizes vertices at pinchpoints to punctured spheres, although we did not realize at the time that it only works for oriented embeddings.
The problem with embeddings that are not oriented is that without a global clockwise direction to force a choice, there are two possible ways to glue a punctured sphere to each of its boundary components, giving many topologically distinct potential duals.

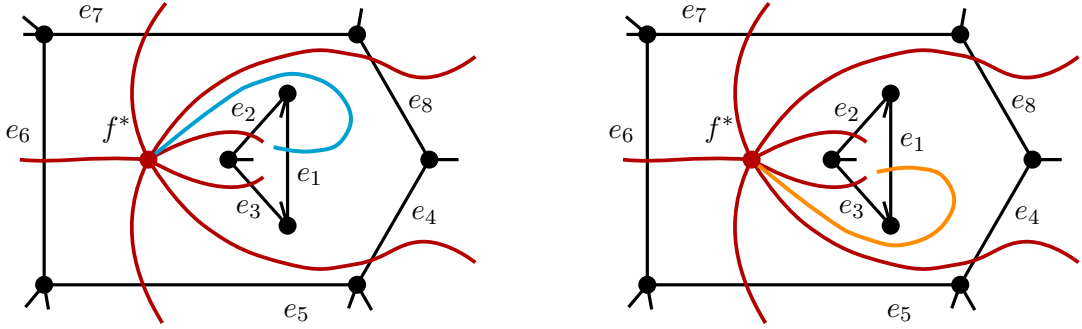
\begin{figure}[ht]
  \hbox to \hsize{\hss%
\newcommand{\figsixseven}[3]{%
\begin{tikzpicture}[scale=0.92]
  \coordinate (e6) at (0,0);
  \coordinate (e7) at (0,3.6);
  \coordinate (e8) at (4.5,3.6);
  \coordinate (e5) at (4.5,0);
  \coordinate (e4) at (5.5392,1.8);
  \coordinate (e1) at (3.5,2.75);
  \coordinate (e2) at (2.65,1.8);
  \coordinate (e3) at (3.5,0.85);   
  \coordinate (H)  at (1.5,1.8);         

  \draw[blk] (-0.3,3.6) -- (e8) -- (e4) -- (e5) -- (e6) -- (e7);
  \draw[blk] (e7) -- ++(140:0.40);
  \draw[blk] (e8) -- ++(80:0.38);
  \draw[blk] (e4) -- ++(0:0.42);
  \draw[blk] (e5) -- ++(-60:0.42);
  \draw[blk] (e5) -- ++(-110:0.36);
  \draw[blk] (e6) -- ++(-130:0.42);
  \draw[blk] (e6) -- ++(-80:0.36);

  \draw[blk] (e1) -- (e2) -- (e3) -- cycle;
  \draw[blk] (e1) -- ++(-106:0.36);
  \draw[blk] (e3) -- ++( 106:0.36);
  \draw[blk] (e2) -- ++(   0:0.36);

  \draw[red] (H) .. controls (1.05,2.70) and (1.30,3.50) .. (1.75,4.05);   
  \draw[red] (H) .. controls (1.05,0.90) and (1.30,0.10) .. (1.70,-0.55);  
  \draw[red] (H) .. controls (0.60,1.86) and (0.20,1.74) .. (-0.35,1.80);  
  \draw[red] plot[smooth,tension=0.8] coordinates
     {(H) (2.30,2.75) (3.50,3.30) (4.50,3.30) (5.45,2.98) (6.20,3.28)};      
  \draw[red] plot[smooth,tension=0.8] coordinates
     {(H) (2.30,0.85) (3.50,0.30) (4.50,0.30) (5.45,0.62) (6.20,0.32)};      
  \draw[red] (H) .. controls (2.10,2.15) and (2.80,2.35) .. (3.15,2.05);
  \draw[red] (H) .. controls (2.10,1.45) and (2.80,1.25) .. (3.15,1.55);

  #1

  \node[uv] at (e6) {}; \node[uv] at (e7) {}; \node[uv] at (e8) {};
  \node[uv] at (e5) {}; \node[uv] at (e4) {}; \node[uv] at (e1) {};
  \node[uv] at (e2) {}; \node[uv] at (e3) {};
  \node[urv, label={above left:$f^*$}] at (H) {};

  \node[vlbl] at (0.68,3.90)  {$e_7$};
  \node[vlbl] at (5.42,2.60)  {$e_8$};
  \node[vlbl] at (5.47,1.00)  {$e_4$};
  \node[vlbl] at (3.63,-0.40) {$e_5$};
  \node[vlbl] at (-0.37,2.17) {$e_6$};
  #3
  \node[vlbl] at (2.944,1.088)  {$e_3$};
\end{tikzpicture}%
}

\newcommand{\colouredarc}{plot[smooth,tension=0.8] coordinates
     {(1.5,1.8) (2.45,2.55) (3.10,2.92) (3.80,3.00) (4.35,2.60) (4.30,2.10) (3.85,1.92) (3.30,1.98)}}

\hbox to 0.95\hsize{%
%
\figsixseven{%
  \draw[tealarc] \colouredarc;
}{Fig 7}{%
  \node[vlbl] at (2.872,2.484)  {$e_2$};
  \node[vlbl] at (3.82,1.55)  {$e_1$};}
\hss
%
\figsixseven{%
  \begin{scope}[cm={1,0,0,-1,(0,3.6)}]   
    \draw[gold] \colouredarc;
  \end{scope}
}{Fig 6}{%
  \node[vlbl] at (2.872,2.484)  {$e_2$};
  \node[vlbl] at (3.82,2.05)  {$e_1$};}
}
    \hss%
  }
  \caption{Different dual embeddings for noncellular face}\label{fig:diffdual}
\end{figure}

One issue with the above construction, even for embeddings in oriented surfaces, is that the dual embedding has a different pseudosurface from the primal embedding.  So it does not provide a natural way to draw the primal and dual graphs simultaneously in the same pseudosurface.
It is possible to embed the underlying graph of the dual embedding into the pseudosurface of the primal embedding, so that each dual edge crosses the corresponding primal edge, but the embedding of the dual graph is in general not unique.  
As a simple example, consider the primal (black) embedding shown in Figure~\ref{fig:diffdual}, with a face $f$ that is an annulus (i.e., cylinder, or sphere with two punctures).  In the dual embedding $f^*$ would be a pinchpoint of pinchdegree $2$, with the cyclic order of the dual edges around $f^*$ being
$(e_1^* e_2^* e_3^*)$ on one plate and $(e_4^* e_5^* e_6^* e_7^* e_8^*)$ on the other.  However, there are different embeddings of the dual (mostly red) graph in the primal surface, two of which appear in the figure.
At left the order of the dual edges around $f^*$ is
$(e_1^* e_2^* e_3^* e_4^* e_5^* e_6^* e_7^* e_8^*)$,
while at right the order is
$(e_2^* e_3^* e_1^* e_4^* e_5^* e_6^* e_7^* e_8^*)$.

Deneen, Shute, and Thomborson \cite[pp.~354--355]{DST91} also briefly consider a second, more general, model of embeddings that includes nonorientable situations, where duality is obtained by exchanging vertices and windows (faces).
However, this model is presented as an abstract data structure. While cellular embeddings of graphs in pseudosurfaces can be represented in this model, they do not discuss whether or how the general objects in this model can be interpreted topologically.
In our terms, they essentially consider a gem, where an origin is a v-gon and a vertex is a set of origins, while a pane is an f-gon and a window (face) is a set of panes.
They describe this in algebraic language, following the approach of Tutte \cite{Tut73,Tut84} for representing embeddings.
Duality exchanges v-gons (origins) with f-gons (panes) and vertices with windows.

A 2015 paper by one of us (J.E.-M.) and Moffatt \cite{EM15} defines deletion and contraction for classic embeddings of graphs in pseudosurfaces.  The idea was to find a complete recursive definition for the Las Vergnas polynomial of embedded graphs.  This is not possible working only with cellular embeddings of graphs in surfaces, so a more general setting, using embeddings in pseudosurfaces, was considered.
The key condition \cite[p.~99]{EM15} is that the minor operations for the embedded graph should correspond to the minor operations for the abstract underlying graph, which as we have noted is not true for the BR minor operations.
The solution given is to use the natural operations of removal for edge deletion and quotienting for edge contraction.  Removal can produce embeddings that are not face-cellular, but the class of embeddings here does not require face-cellularity.
This paper works with an abstract graph dual of the embedded graph, but does not try to define a dual embedding.
The abstract graph dual is determined by the stabilization class of the embedding, and the minor operations used here also preserve the stabilization class, so we can interpret the results of \cite{EM15} as applying to stabilization classes of classic embeddings.

A 2018 paper of Krajewski, Moffatt, and Tanasa \cite{KMT18} sets up a general Hopf algebra framework where a `canonical' Tutte polynomial can be defined in terms of deletion and contraction operations.  A number of examples are considered, three of which are relevant to pseudosurface embeddings.  Note that we simplify the results of \cite{KMT18} slightly: the results apply to the objects we describe up to equivalence relations of a `matroidal' nature (for example, isolated vertices are usually ignored).

The first relevant example \cite[p.~278 and \S4.4]{KMT18} is based on the results from \cite{EM15} for Las Vergnas polynomials of classic embeddings in pseudosurfaces, up to stabilization.
The second and third examples are phrased in terms of vertex-partitioned graphs.
The second example  \cite[p.~279 and \S4.7]{KMT18} considers extended \bollobas.-Riordan polynomials for vertex-partitioned ribbon graphs, which correspond to cellular embeddings in a pseudosurface.
In terms of pseudosurface embeddings, the deletion and contraction operations discussed are BR deletion and quotienting contraction.
As noted in \cite[Remark 62]{KMT18}, vertex-partitioned ribbon graphs seem to be the natural extended setting for the \bollobas.-Riordan polynomial of ribbon graphs.
The third \cite[p.~280 and \S4.8]{KMT18} considers extended Krushkal polynomials for vertex-partitioned graphs embedded in surfaces (with no face restrictions), up to stabilization.
As pointed out in \cite[p.~11]{MS18}, vertex-partitioned surface embeddings correspond to classic embeddings of graphs in pseudosurfaces.
Considered in terms of pseudosurface embeddings, the deletion and contraction operations are just the natural removal and quotienting, respectively.  So this is essentially the same setting as the first example, based on \cite{EM15}, but the extended Krushkal polynomial provides more information than the Las Vergnas polynomial.

The 2020 paper \cite{HM20} of Huggett and Moffatt investigates both duality and minors for embeddings of graphs in surfaces in a systematic way.
They consider classic embeddings of graphs in pseudosurfaces.
However, they obtain a consistent duality theory by moving to a more general class of objects.  These are represented in two ways: first as colored ribbon graphs, and second as stabilization classes of graph embeddings in pseudosurfaces.
A \emph{colored ribbon graph} assigns colors to the vertices, the boundaries (representing faces), or both, of a ribbon graph.  The colors serve to establish a partition or equivalence relation on the vertices or boundaries.  Duality exchanges vertices with boundaries, with the color of a primal vertex becoming the color of its dual boundary, and vice versa.
Using the direct correspondence between ribbon graphs and gems, this is equivalent to partitioning the set of v-gons, the set of f-gons, or both, in a gem, and duality swaps v-gons and f-gons while preserving the partitions.  In other words, this is the same as the second model of Deneen, Shute, and Thomborson \cite{DST91}, discussed above.

Each face-supported classic graph embedding in a pseudosurface has an associated colored ribbon graph.  We discard unsupported faces, cellularize the faces, and dismantle the vertices to obtain a cellular embedding in a surface, corresponding to a ribbon graph.  Then we partition the vertices according to the original vertices they were split from, and we partition the faces (boundaries) according to the face whose boundary component they capped.
However, many different embeddings of this type correspond to a given colored ribbon graph: these form a stabilization class of face-supported classic embeddings.
This approach is reminiscent of Deneen, Shute, and Thomborson's first model, using punctured spheres as faces: a classic embedding with all faces punctured spheres is just a particularly simple representative of its stabilization class.
Duality here is most easily understood by passing to the associated colored ribbon graphs.

A significant feature of \cite{HM20} is the systematic examination of four possible models for graphs embedded in pseudosurfaces or surfaces.
Starting with classic embeddings, we can add the restriction that they be face-cellular or vertex-cellular or both.
Each model requires different edge deletion and contraction operations, to make sure the class of objects is closed under taking minors.  Each model has its own associated Tutte polynomial defined using its minor operations.
The four models are as follows (with numbering added by us, corresponding to the first four rows of Table \ref{tab:subclass} in Subsection \ref{ss:subclass}).
\begin{enumerate}[(1),nosep]
\item\label{modfour} Vertex-cellular cellular embeddings, i.e., cellular embeddings in surfaces, with BR deletion and BR contraction.  These correspond to (uncolored) ribbon graphs (both vertices and boundaries always have the discrete partition).
\item\label{modtwo} Cellular embeddings in pseudosurfaces, with BR deletion (removal then cellularization of faces) and quotienting for contraction.
These correspond to vertex-colored ribbon graphs (only vertices are colored; boundaries always have the discrete partition).
\item\label{modthree} Vertex-cellular classic embeddings (considered up to stabilization), i.e., not necessarily cellular embeddings in surfaces, with removal for edge deletion and BR contraction (quotienting then dismantling of vertices).  These correspond to boundary-colored ribbon graphs (only boundaries are colored; vertices always have the discrete partition).
\item\label{modone} Classic embeddings in pseudosurfaces (considered up to stabilization), with removal for edge deletion and quotienting for edge contraction.  These correspond to colored ribbon graphs (both vertices and boundaries are colored).
\end{enumerate}
Models \ref{modfour} and \ref{modone} are closed under duality.  Models \ref{modtwo} and \ref{modthree} are dual to each other.  In models \ref{modthree} and \ref{modone}, unsupported faces may exist but must be surfaces, which can be ignored due to stabilization.

The minor operations for each model can also be defined as operations on colored ribbon graphs.  We use BR operations for the ribbon graphs, but then adjust the colorings.  If vertices are colored, as in models \ref{modtwo} and \ref{modone}, an operation that splits a vertex into two vertices assigns both new vertices to the color of the original vertex; an operation that combines two vertices merges their color classes into a new color.  If boundaries (faces) are colored, as in models \ref{modthree} and \ref{modone}, an operation that splits a boundary into two boundaries assigns both new boundaries to the color of the original boundary; an operation that combines two boundaries merges their color classes into a new color.
These definitions are easily translated to gems or jewels with partitions of the v-gons ($\ca\cv$-cycles), and/or of the f-gons ($\ca\cf$-cycles).  We perform the usual deletion or contraction and then modify the partitions (interpreted as colorings) as described.

The derivation and properties of the Tutte polynomials for the four models are discussed in \cite[Sections 3 to 5]{HM20}, in terms of colored ribbon graphs.  The extended Krushkal polynomial from \cite{KMT18} mentioned above (and hence the Las Vergnas polynomial from \cite{EM15}) only depends on the stabilization class of an embedding, and can be related to one of these Tutte polynomials.  The extended \bollobas.-Riordan polynomial from \cite{KMT18} mentioned above can also be related to an appropriate Tutte polynomial.

Moffatt and Thompson \cite{MT24, Thom26} have extended the idea of colored ribbon graphs to \emph{packaged ribbon graphs}, which are colored ribbon graphs in which each color class of vertices or boundaries is assigned a nonnegative integer weight.  Operations on these are defined as for colored ribbon graphs, but with additional rules to handle the weights.  Using packaged ribbon graphs they are able to provide Tutte polynomials with several nice properties for embeddings of graphs in pseudosurfaces, 
which extend Tutte polynomials due to two subsets of Goodall, Krajewski, Litjens, Regts, and Vena \cite{GKRV18,GLRV20} for embeddings of graphs in surfaces.

\section{Quasicellular and pseudocellular embeddings}

In this section we define our main concept, pseudocellular embeddings of graphs in pseudosurfaces. We first define a subclass of pseudocellular embeddings known as quasicellular embeddings, which encompass some of the previous approaches described in Subsection \ref{ss:prior}, and discuss duality and minors for this subclass. Then we define pseudocellular embeddings, and describe how other objects related to graph embeddings, such as edge-point ribbon graphs and cogs (cyclically ordered graphs) also correspond to subclasses of pseudocellular embeddings.

\subsection{Quasicellular embeddings}
\label{ss:quasicellular}
Previous attempts at defining duality for embeddings in pseudosurfaces \cite{DST91,HM20} have been successful to some extent.  However, they either work only for oriented surfaces, or have to extend embeddings in pseudosurfaces to more abstract structures: the abstract data structures of the second model in \cite{DST91}, or colored ribbon graphs and stabilization classes of embeddings in \cite{HM20}.
They do not provide an interpretation of the dual of an embedding as a unique embedding in the same pseudosurface as the primal embedding.
It turns out that it is not difficult to define a class of embeddings of graphs in pseudosurfaces, including cellular embeddings, where duality can be implemented to give a unique embedding in the same pseudosurface.  The barrier to doing this seems to have been the fixation on the idea that only classic embeddings, with all pinchpoints at vertices, should be considered.

Since we want our class of embeddings to include all cellular pseudosurface embeddings, let us consider what happens if we try to take the dual of a cellular embedding.  As previously mentioned, there is a unique dual graph with a unique (up to homeomorphism preserving the primal graph) embedding in the primal pseudosurface.
The only issue is that we may now have faces that contain a single pinchpoint at the location corresponding to a primal vertex $v$, i.e., the face $v^*$ is a multidisk neighborhood of $v$.  If we want a dual that is embedded in the primal pseudosurface, we should just accept that such faces are allowed.  If we reverse the dual operation, a multidisk face $f$ gives a vertex located at the center $f^*$ of the face, i.e., at the pinchpoint, if there is one.  Since a face should have a unique dual vertex, we do not want to allow faces that contain more than one pinchpoint.  Thus, we say a face is \emph{pseudocellular} if it is homeomorphic to a multidisk, and an embedding is \emph{face-pseudocellular} if every face is pseudocellular.  Face-pseudocellular embeddings are face-supported.

We now define some terminology that will be useful later.  Suppose we have a pseudocellular face $f$ in an embedding $G$ in a pseudosurface $\Pi$ created by $q : \Si \to \Pi$.
Then $f$ is a $d$-disk for some $d \ge 1$, and $q\iv(f)$ is a union of disjoint open disks $D_1, D_2, \dots, D_d \sb \Si$. We call the sets $P_i = q(D_i)$ the \emph{plates of $f$}, or in general \emph{face plates}.  We have $P_i \cap P_j = \{f\du\}$ for $i \ne j$.

An embedding is \emph{quasicellular} if it is face-pseudocellular and edge-cellular.
Quasicellular embeddings of graphs in pseudosurfaces are then in one-to-one correspondence with the abstract data structures in the second model of \cite{DST91}, and with the colored ribbon graphs or stabilization classes of classic embeddings from \cite{HM20}.
We can start with a quasicellular embedding, dismantle both the vertices and the centers of the faces to obtain a cellular embedding, take its ribbon graph or gem, and find the appropriate partition of vertices/v-gons and faces (boundaries)/f-gons based on what was dismantled to give a particular cellular vertex or face.  This gives a colored ribbon graph or an abstract data structure from the second model in \cite{DST91}.
Or starting with such an embedding, we can replace each pinched face with a surface with boundary, giving a classic embedding, and we can take its stabilization class.
Another way to think of this is that if we consider stabilization classes of edge-cellular embeddings in pseudosurfaces, a larger set than classic embeddings because it allows pinchpoints in faces, then the face-pseudocellular embeddings give a unique representative for each stabilization class.  Therefore, considering `arbitrary' (meaning edge-cellular) embeddings up to stabilization is equivalent to considering face-pseudocellular edge-cellular, i.e., quasicellular, embeddings.

The dual of a quasicellular embedding is defined in the obvious way: put a dual vertex at the center of each face, and cross each primal edge $e$ with a dual edge $e^*$ joining the dual vertices representing the faces on either side of $e$.  This is unique up to homeomorphisms that preserve the primal graph.  It is easy to see that this is equivalent to the dual operation for colored ribbon graphs.

In a quasicellular embedding a loop can cross its vertex terminally.  In this case we cannot classify $e$ as twisted or untwisted, and we just say it is a \emph{terminal-crossing loop}.  Other edges can be classified as links, twisted loops or untwisted loops as for cellular embeddings in surfaces.  An edge may also be a \emph{terminal-crossing coloop}, which means it has the same face, but different plates, on its two sides.  Other edges can be classified as colinks, twisted coloops, or untwisted coloops.

For a quasicellular embedding, connectedness of the pseudosurface does not guarantee that the embedded graph is connected.  In fact, we can have any combination of primal connected or disconnected and dual connected or disconnected.

\subsection{Quasicellular minors and Q operations versus BR operations}\label{ss:qbr}
For edge-based minors of quasicellular embeddings, we consider contraction first.  We contract an edge $e$ using the natural quotienting operation, crushing $e$ to a single point.
Since an edge $e$ is either a simple path or a simple closed curve, by Theorem \ref{curvecontract} identifying $e$ to a single point keeps the underlying space a pseudosurface.
Moreover, by Observation \ref{cifo}, the topology at points that do not belong to $e$ essentially does not change, so faces are still pseudocellular, edges other than $e$ remain cellular, and we still have a quasicellular embedding.
This operation acts as edge contraction on the underlying graph.  See Figure~\ref{fig:concelledge}.

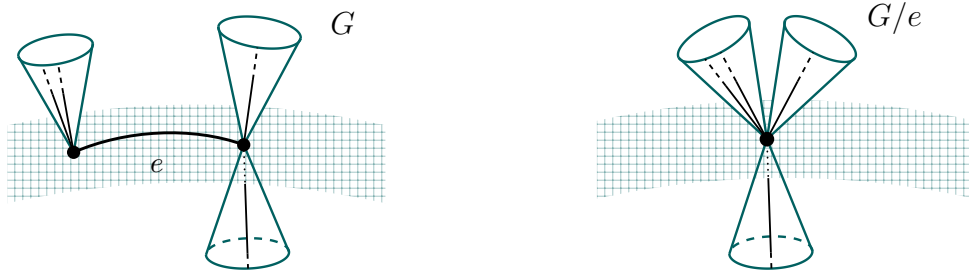
\begin{figure}[ht]
  \hbox to \hsize{\hss%

\begin{minipage}[t]{0.48\textwidth}\centering
\begin{tikzpicture}[scale=1.25]
  \narrowgoldband 
  \coordinate (L) at (-1.30,0.482);
  \coordinate (R) at ( 0.50,0.561);

  \maskcone{-1.30}{0.482}{100}{1.10}{0.40}{0.15}
  \drawcone{-1.30}{0.482}{100}{1.10}{0.40}{0.15}{solid}
  \coneaxisup{-1.30}{0.482}{100}{1.10}{0.15}
  \coneray{-1.30}{0.482}{100}{1.10}{0.40}{0.15}{154.87}
  \maskcone{0.50}{0.561}{82}{1.20}{0.44}{0.16}
  \drawcone{0.50}{0.561}{82}{1.20}{0.44}{0.16}{solid}
  \coneaxisup{0.50}{0.561}{82}{1.20}{0.16}
  \drawcone{0.50}{0.561}{-88}{1.15}{0.44}{0.16}{dashed}
  \coneaxisdown{0.50}{0.561}{-88}{1.15}{0.44}{0.16}{0.4133}

  \draw[edgeline] (L) .. controls (-0.75,0.72) and (-0.05,0.76) .. (R);
  \fill (L) circle (2.0pt);
  \fill (R) circle (2.0pt);

  \node[slbl] at (-0.42,0.34) {$e$};
  \node[lbl]  at ( 1.55,1.85) {$G$};
\end{tikzpicture}
\end{minipage}
\hfill
\begin{minipage}[t]{0.48\textwidth}\centering
\begin{tikzpicture}[scale=1.25]
  \narrowgoldband 
  \coordinate (V) at (-0.20,0.573);

  \maskcone{-0.20}{0.573}{118}{1.20}{0.42}{0.16}
  \drawcone{-0.20}{0.573}{118}{1.20}{0.42}{0.16}{solid}
  \coneaxisup{-0.20}{0.573}{118}{1.20}{0.16}
  \coneray{-0.20}{0.573}{118}{1.20}{0.42}{0.16}{154.72}
  \maskcone{-0.20}{0.573}{62}{1.20}{0.42}{0.16}
  \drawcone{-0.20}{0.573}{62}{1.20}{0.42}{0.16}{solid}
  \coneaxisup{-0.20}{0.573}{62}{1.20}{0.16}
  \drawcone{-0.20}{0.573}{-88}{1.20}{0.44}{0.17}{dashed}
  \coneaxisdown{-0.20}{0.573}{-88}{1.20}{0.44}{0.17}{0.4126}

  \fill (V) circle (2.2pt);
  \node[lbl] at (1.15,1.86) {$G/e$};
\end{tikzpicture}
\end{minipage}
    \hss%
  }
  \caption{Contracting a cellular edge}\label{fig:concelledge}
\end{figure}

However, for edge deletion the natural removal operation can cause problems.
It may still create a face (or pinchcomponent of a face) homeomorphic to a cylinder or \mobius. strip, but there is another issue.  Removing a colink can merge two faces, and if both faces are pinched we then have a face with two pinchpoints.
To return to a face-pseudocellular embedding we need to merge these two pinchpoints.  This can be done by contracting (i.e., crushing) the dual edge joining the centers of the two faces.  Contracting the dual edge also reduces a face with a pinchcomponent that is a cylinder or \mobius. strip to a multidisk, so it is a general solution to our problems.
Thus, edge deletion consists of removal followed by contraction of the dual edge.
See Figure~\ref{fig:delcelledge}.
In other words, we define edge deletion as the dual operation to edge contraction: $G\dt e = (G\du \ct e\du)\du$.  This preserves the fact that we have a quasicellular embedding, and acts as edge deletion on the underlying graph.

\begin{figure}[ht]
  \hbox to \hsize{\hss%

\begin{minipage}[t]{0.48\textwidth}\centering
\begin{tikzpicture}[scale=1.25]
  \narrowgoldbanddeep
  \coordinate (L) at (-1.50,0.451);
  \coordinate (R) at ( 0.30,0.570);

  \maskcone{-1.50}{0.451}{100}{1.05}{0.38}{0.14}
  \drawcone{-1.50}{0.451}{100}{1.05}{0.38}{0.14}{solid}
  \maskcone{0.30}{0.570}{80}{1.05}{0.38}{0.14}
  \drawcone{0.30}{0.570}{80}{1.05}{0.38}{0.14}{solid}

  \draw[redline] (L) .. controls (-0.95,0.70) and (-0.25,0.76) .. (R);

  \draw[edgeline] (-0.62,0.28) .. controls (-0.534,0.475) and (-0.520,0.675)
                                .. (-0.58,0.88);
  \blackframe{0}

  \fill[red!65!black] (L) circle (2.3pt);
  \fill[red!65!black] (R) circle (2.3pt);
  \node[slbl] at (-0.38,0.47) {$e$};
  \node[slbl,red!65!black] at (-1.06,0.80) {$e^{*}$};

\end{tikzpicture}
\end{minipage}
\hfill
\begin{minipage}[t]{0.48\textwidth}\centering
\begin{tikzpicture}[scale=1.25]
  \narrowgoldbanddeep
  \coordinate (V) at (-0.40,0.566);

  \maskcone{-0.40}{0.566}{135}{1.15}{0.40}{0.15}
  \drawcone{-0.40}{0.566}{135}{1.15}{0.40}{0.15}{solid}
  \maskcone{-0.40}{0.566}{45}{1.15}{0.40}{0.15}
  \drawcone{-0.40}{0.566}{45}{1.15}{0.40}{0.15}{solid}

  \blackframe{0.17}
  \fill[red!65!black] (V) circle (2.6pt);
\end{tikzpicture}
\end{minipage}
    \hss%
  }
  \caption{Deleting a cellular edge}\label{fig:delcelledge}
\end{figure}
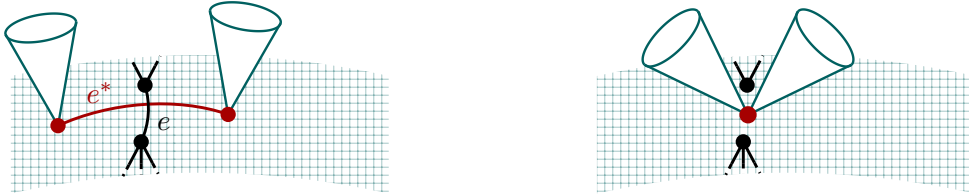

We call our minor operations the \emph{Q minor} operations because they are defined simply by topological quotienting.  The Q minor operations correspond directly to operations on the underlying graph, unlike the BR minor operations discussed in Subsection \ref{ss:minorsurf}.  For the BR operations deletion (in ribbon graphs) is the primary operation, and contraction is defined as the dual of deletion; for the Q operations contraction (in the actual embedding) is the primary operation, and deletion is defined as the dual of contraction.  

It is a routine exercise to follow the correspondence between quasicellular embeddings and ribbon graphs described in Subsection \ref{ss:quasicellular} to check that the Q operations correspond exactly to the edge deletion and contraction operations for colored ribbon graphs in \cite{HM20}, which are equivalent to the minor operations of \cite{EM15} and \cite[Subsections 4.4 and 4.8]{KMT18}.  This can be done in two steps, using the translation of Q minor operations into operations on jewels with v-gon and f-gon partitions, which is described in greater generality in Subsection \ref{ss:pseudocellular} below, and then using the correspondence between jewels with v-gon and f-gon partitions and colored ribbon graphs.

The Q operations do not preserve vertex- or face-cellularity: Q edge contraction can create vertex pinchpoints, and Q edge deletion can create facial pinchpoints.  We need to modify the Q operations if we wish to stay in a vertex-cellular or face-cellular class of embeddings.
To stay in a class of vertex-cellular embeddings, we dismantle any new vertex pinchpoint created by Q edge contraction, which gives BR edge contraction.  To stay in a class of face-cellular embeddings, we dismantle any new facial pinchpoint created by Q edge deletion, which is equivalent to our previous definition of BR edge deletion, namely edge removal followed by cellularization.

Huggett and Moffatt \cite{HM20} provide four embedding models and emphasize that each model needs different minor operations.  
Their four models can be regarded as describing quasicellular embeddings that are vertex-cellular, face-cellular, both, or neither.  Their different minor operations are exactly what we suggest here: use a Q operation when vertex- or face-cellularity does not need to be preserved, and a BR operation when it does.
In the face-cellular case these operations agree with the minors of \cite[Subsection 4.7]{KMT18}.

These rules extend to other operations that we will define later: we can define Q operations involving quotienting, and then define BR operations where we preserve vertex-cellularity by adding a step to dismantle a new vertex pinchpoint, or face-cellularity by adding a step to dismantle a new facial pinchpoint.

So at this point we have a set of objects, quasicellular embeddings of graphs in pseudosurfaces, with a theory of duality and edge-based minors equivalent to the second model of \cite{DST91} and to the models of \cite{HM20}.  Thus, all of the polynomials from \cite{EM15,HM20,KMT18} that we have discussed can be described in terms of these embeddings.
Unlike the models of \cite{DST91,HM20}, our model works with concrete embeddings in pseudosurfaces, the dual of an embedding is defined naturally as an embedding in the same pseudosurface as the primal, and our operations are defined using simple topological operations of quotienting and dual quotienting, with dismantling of pinchpoints in some models.

But we are not finished.  It turns out that quasicellular embeddings can be extended to allow pinchpoints on edges, and that using this extension we can model other objects related to embeddings of graphs.  We address this in the next subsection.

\subsection{Pseudocellular embeddings}\label{ss:pseudocellular}
Quasicellular embeddings are edge-cellular, but we did not justify this particular restriction.  In this subsection we consider how we can relax it to allow pinchpoints in edges.  We still want to be able to dualize vertices to pseudocellular faces, so we still consider only face-pseudocellular embeddings.

We do want some restrictions on pinchpoints contained in the interior of an edge.  To form a dual graph, an edge should be incident with (contained in the closure of) no more than two distinct faces.  If an edge internally skims a pinchpoint $p$, it would generally be incident with at least three faces: two faces in the plate at $p$ in which the edge skims $p$, and one in another plate.  If an edge internally crosses two pinchpoints, $p_1$ then $p_2$, then the three segments of the edge (before $p_1$, between $p_1$ and $p_2$, and after $p_2$) would in general be incident with different faces, so the whole edge would be incident with three or more faces.

This suggests that if we want a class of embeddings of graphs in pseudosurfaces that can be dualized, the only reasonable possibility for a noncellular edge is a \emph{cross-pinched} edge, whose interior contains exactly one pinchpoint of pinchdegree $2$, which the edge crosses.
An edge is \emph{pseudocellular} if it is either cellular or cross-pinched.
The \emph{center} of a pseudocellular edge is the pinchpoint, if it is cross-pinched, and an arbitrary interior point, if it is cellular.
An embedding is \emph{edge-pseudocellular} if every edge is pseudocellular.  Our main class of embeddings in pseudosurfaces is specified as follows.

\begin{definition}\label{pseudocellular}
An embedding of a graph in a pseudosurface is \emph{pseudocellular} if it is face-pseudocellular and edge-pseudocellular.  This means that all faces are homeomorphic to open multidisks, and pinchpoints can occur in three ways: at a vertex, as the only pinchpoint in a given face, or as the only pinchpoint in the interior of a given edge, which must have pinchdegree $2$ and be crossed by the edge.
A quasicellular embedding is precisely an edge-cellular pseudocellular embedding.
(Note that the term was \emph{pseudocellular} was used in \cite[Chapter 6]{Dun20} with a different meaning.)
\end{definition}

There is a useful correspondence between cross-pinched edges and semiedges.  We can easily include semiedges in quasicellular and pseudocellular embeddings by generalizing Subsection \ref{ss:embgraphsurf}.  Removing and crushing a semiedge are equivalent.
Given two semiedges, we obtain a cross-pinched edge by identifying their tips.
Conversely, given a cross-pinched edge, we can obtain two semiedges by dismantling its center.
These inverse operations do not change the faces of the embedding.
When tracing a boundary component of a face, we traverse a cross-pinched edge by following one of its half-edges and returning along the same half-edge, which is exactly what we would do if the half-edge was just a semiedge.
If we dismantle the centers of all cross-pinched edges of a pseudocellular embedding, we obtain a quasicellular embedding with semiedges.

The correspondence between cross-pinched edges and semiedges provides a representation for pseudocellular embeddings using gems (or jewels).  As we saw in Subsection \ref{ss:quasicellular}, quasicellular embeddings can be represented by gems with partitioned v-gons and f-gons.  We can represent pseudocellular embeddings, possibly with semiedges, by gems that also have partitioned semiedge-digons, where each part contains only one or two semiedge-digons.  Parts with one semiedge-digon represent semiedges, and parts with two semiedge-digons represent a cross-pinched edge obtained by identifying the tips of the two semiedges.  Cellular edges are represented by e-squares, as usual.

The representation for gems and jewels can also be translated into the language of colored ribbon graphs.  However, there is another perspective that is sometimes useful.  Edge ribbons in a ribbon graph are attached to part of the boundary of a vertex, and transmit relative orientation information between the vertices at the ends of the edge.  For a cross-pinched edge we still want to specify a location on the boundary of each vertex, but we do not want to transmit orientation information.  We can do this, for example, by using a $1$-dimensional \emph{string} to represent a cross-pinched edge, instead of a $2$-dimensional ribbon (the idea of using strings instead of ribbons comes from \cite[\S3.2.1]{EM13}).  So pseudocellular embeddings can be represented as (face- and vertex-) colored ribbon/string graphs.  We ignore strings when determining boundary components in a ribbon/string graph.

In terms of band decompositions, a pseudocellular embedding corresponds to a band decomposition of a pseudosurface into $0$-, $1$-, and $2$-bands, where all bands are closed multidisks.
The alternating property described in Subsection \ref{ss:repembsurf} holds around each boundary component of each $i$-band.
The $1$-bands are still restricted to have exactly four boundary segments, so they must either be disks with four boundary segments, or $2$-disks where each plate has two boundary segments.

Pseudocellular embeddings are dualized by the same natural procedure we use for cellular surface embeddings and quasicellular embeddings. We put a dual vertex at the center of each face, and cross each primal edge $e$ with a dual edge $e^*$ joining the dual vertices representing the faces on either side of $e$.
Each dual edge $e^*$ crosses $e$ at the center of $e$, which is also designated as the center of $e^*$.
If $e$ is cross-pinched then $e^*$ is also cross-pinched.  The dual is therefore embedded in the same pseudosurface as the primal.  See Figure~\ref{fig:dualcrpnedge} for an example of the dual of a cross-pinched edge: the (black) horizontal primal edge crossing the pinchpoint is dual to the (red) slanted dual edge crossing the pinchpoint.

\begin{figure}[ht]
  \hbox to \hsize{\hss%

\begin{tikzpicture}[scale=1.25]
  \drawconeT{0}{0}{180}{2.40}{1.6805}{0.61}{solid}
  \drawconeT{0}{0}{0}{2.40}{1.6805}{0.61}{solid}

  \coneslice{ 1.31388}{-0.17408}{-73.78}{0.96309}{0.2737}
  \coneslice{-1.31388}{-0.17408}{ 73.78}{0.96309}{0.2737}

  \draw[edgeline] (-0.45000,0) -- (0.45000,0);

  

  \draw[edgeline] (-0.45000,0) -- (-1.79210, 0.13935);
  \draw[edgeline] (-0.45000,0) -- (-2.05993,-1.39513);
  \draw[edgeline] ( 0.45000,0) -- ( 1.80669,-0.39041);
  \draw[edgeline] ( 0.45000,0) -- ( 1.85546, 0.75734);  


  \draw[redline] (-1.82459, 0.55783) -- ( 1.82459,-0.55783);

  \fill (-0.45000,0) circle (2.3pt);
  \fill ( 0.45000,0) circle (2.3pt);
  \fill[red!65!black] (-0.94473, 0.28883) circle (2.3pt);
  \fill[red!65!black] ( 1.07481,-0.32860) circle (2.3pt);
  \fill[red!65!black] (-1.06622,-0.30112) circle (2.3pt);
  \fill[red!65!black] ( 0.97886, 0.04916) circle (2.3pt);

  \node[lbl] at (-2.60,-2.10) {Fig 4.};
\end{tikzpicture}
    \hss%
  }
  \caption{A cross-pinched edge and its dual}\label{fig:dualcrpnedge}
\end{figure}

We can also interpret this dual using semiedges.  Semiedges can be dualized in quasicellular embeddings by generalizing Subsection \ref{ss:dualsurf}.
To dualize a pseudocellular embedding we dismantle centers of cross-pinched edges to obtain a quasicellular embedding with semiedges, take the dual, and then re-identify the tips of each pair $d_1^*, d_2^*$ of dual semiedges that are duals of a pair $d_1, d_2$ of semiedges obtained by dismantling the center of a cross-pinched edge.  We obtain a pseudocellular embedding in the same pseudosurface, because the final re-identifications just reverse the initial dismantlings.

In a pseudocellular embedding we can still classify cellular edges as (co)links, or as twisted, untwisted, or terminal-crossing (co)loops.  However, cross-pinched edges and their duals cannot transmit orientation information, so we must classify them as just (co)links, non-terminal-crossing (co)loops, or terminal-crossing (co)loops.

Edge deletion and edge contraction for pseudocellular embeddings extend the Q minor operations for quasicellular embeddings.  To contract an edge we just crush it to a point, and the result is still a pseudocellular embedding by similar arguments to those used for quasicellular embeddings.  This corresponds to contracting $e$ in the underlying graph.  To delete an edge we remove it and contract (i.e., crush) its dual edge, so that $G\dt e = (G\du\ct e\du)\du$, and again we still have a pseudocellular embedding.  This corresponds to deleting $e$ in the underlying graph.

As with quasicellular embeddings, we can modify these Q operations to obtain BR operations that allow us to preserve vertex- or face-cellularity.  To stay in a class of vertex-cellular pseudocellular embeddings we dismantle any vertex pinchpoints after Q edge contraction.
To stay in a class of face-cellular pseudocellular embeddings we dismantle any pinchpoints at the centers of faces after Q edge deletion.

We can interpret the Q minor operations for a cross-pinched edge $e$ of a pseudocellular embedding $G$ in various ways, but first we define \emph{erasure} of $e$ to mean removing $e$ and then dismantling the pinchpoint at the center of $e$.  By the techniques from Section \ref{sec:topology} this is equivalent to first dismantling the center of $e$, so that $e$ becomes two semiedges, then either removing or crushing the semiedges.  The result, which we denote by $G\er e$, is still a pseudocellular embedding.  We note that $(G\er e)\du = G\du \er e\du$.

For a cross-pinched edge $e$ we can describe contraction or deletion of $e$ by considering semiedges.  To contract $e$, we can dismantle the center of $e$, crush each of the two resulting semiedges into a point (an endvertex of $e$), which has no effect on the overall topology, and then identify the vertices at the ends of $e$ (which may already be equal).  The result is the same as erasing $e$, and then identifying the vertices of $e$.
In a similar way, deleting $e$ corresponds to erasing $e$, and then identifying the centers of the faces to which $e$ was incident (which may already be equal).

There is another operation we can define on a pseudocellular embedding $G$, which turns a cellular edge $e$ into a cross-pinched edge.  To do this we remove the middle third of $e$, turning $e$ into two separate semiedges $e_1, e_2$.  As with removing an edge, this may create a face (or pinchcomponent of a face) that is a cylinder or a \mobius. strip, or a face with two pinchpoints, so we then contract the original dual edge $e^*$, which resolves these possible problems.
Finally we identify the tips of $e_1$ and $e_2$ to form the new version of $e$, crossing a new pinchpoint.  We call this operation \emph{pinching} $e$, and denote it by $G \pn e$; it merges the two faces on either side of $e$ if they are distinct, and always turns $e$ into a cross-pinched coloop.
We can also pinch the dual edge $e\du$ to obtain $(G\du \pn e\du)\du$, which we call \emph{copinching} of $e$, and abbreviate to $G \dpn e$.
This operation merges the two endvertices of $e$ if they are distinct, and turns $e$ into a cross-pinched loop.
For a cellular edge $e$ we have $G \pn e \er e = G\dt e$, and $G \dpn e \er e = G\ct e$.
If we are working in a class of face-cellular or vertex-cellular embeddings, we can also use BR versions of pinching and copinching, respectively, which dismantle any new facial or vertex pinchpoints, respectively.

We briefly describe the operations above in terms of gems or jewels.
Erasure of a cross-pinched edge $e$ means smoothing along any one color in each of the paired semiedge-digons corresponding to $e$.  Each original v-gon and f-gon in the $G$ corresponds to a unique new v-gon or f-gon, so the v-gon and f-gon partitions transfer naturally to the new gem or jewel.

Contraction of an edge $e$ (either cellular or cross-pinched) means smoothing along $\cf$ in the e-square or paired semiedge-digons corresponding to $e$.  The f-gons and their partition transfer naturally to the new gem or jewel.  There are one or two new v-gons that came from two (possibly equal) old v-gons; their classes in the partition are merged if necessary into a single class containing the new v-gon or v-gons.
Deletion of $e$ is similar, but we smooth along $\cv$, and exchange the roles of v-gons and f-gons.

Pinching a cellular edge $e$ means deleting the edges not colored $\cv$ from the e-square corresponding to $e$, and then adding the deleted edges back as edges parallel to the edges of color $\cv$, and pairing the created edge-semidigons in the semiedge-digon partition. 
The v-gons, and their partition, do not change.
The two (possibly equal) original f-gons $F_1, F_2$ through the e-square corresponding to $e$ become two (possibly equal) new f-gons $F_1', F_2'$ through the paired semiedge-digons corresponding to $e$.  If necessary, we merge the classes of the f-gon partition containing $F_1$ and $F_2$, and then replace $F_1, F_2$ by $F_1', F_2'$ to obtain the new f-gon partition.
Copinching is similar, but exchanges the roles of $\cv$ and v-gons with $\cf$ and f-gons.

Vertex deletion is usually also considered a minor operation for embedded graphs, and can be regarded as deletion of all incident edges followed by some operation that eliminates the vertex.  For pseudocellular embeddings, even edgeless graphs can have nontrivial structure, and we discuss vertex deletion and related operations later, in Subsection \ref{ss:edgeless}

\subsection{Subclasses of pseudocellular embeddings}%
\label{ss:subclass}
In this subsection we discuss various subclasses of pseudocellular embeddings, and how some of them correspond to objects previously studied in the literature.  To form the subclasses we can impose restrictions on vertices, faces and edges.
We can consider embeddings that are unrestricted as to vertices and faces, vertex-cellular, face-cellular, or both, giving four possibilities.
We can separately decide whether the embedding should be unrestricted as to edges, edge-cellular, or have all edges cross-pinched.
Using these criteria we can define twelve subclasses, listed in Table \ref{tab:subclass}, where a dash represents no added condition beyond the standard restriction for a pseudocellular embedding.

\begin{table}[ht]
\begingroup
\def\cc{cellular}\def\aa{$-$}\def\pp{cross-pinched}
\begin{tabular}{|c|c|c|c|l|}
\hline
& Vertices & Faces & Edges & Notes \cr
\hline
\noalign{\vskip2pt}
\hline
 (1) & \cc & \cc & \cc & cellular embeddings in surfaces \cr
 (2) & \aa & \cc & \cc & cellular embeddings in pseudosurfaces /\cr
     &     &     &     & \quad vertex-partitioned ribbon graphs
				\cite[\S4.7]{KMT18} /\cr
     &     &     &     & \quad vertex-colored ribbon graphs \cite{HM20} \cr
 (3) & \cc & \aa & \cc & dual of cellular embeddings in pseudosurfaces /\cr
     &     &     &     & \quad boundary-colored ribbon graphs or \cr
     &     &     &     & \quad stabilization classes of arbitrary embeddings\cr
     &     &     &     & \quad in surfaces \cite{HM20} \cr
 (4) & \aa & \aa & \cc & DST 2nd model \cite{DST91} / colored ribbon graphs
				or \cr
     &     &     &     & \quad stabilization classes of classic embeddings\cr
     &     &     &     & \quad in pseudosurfaces \cite{HM20} /\cr
     &     &     &     & \quad quasicellular embeddings in pseudosurfaces \cr
\hline
 (5) & \cc & \cc & \aa & edge-point ribbon graphs \cite{EKM18} \cr
 (6) & \aa & \cc & \aa & \cr
 (7) & \cc & \aa & \aa & \cr
 (8) & \aa & \aa & \aa & pseudocellular embeddings in pseudosurfaces \cr
\hline
 (9) & \cc & \cc & \pp & cogs = rigid-vertex graphs
			\cite{BEEMM25plus,BDJSV15,EM13,Kau89},\cr
     &     &     &     & \quad cogs for (1) and (5)\cr
(10) & \aa & \cc & \pp & extended cogs for (2) and (6)\cr
(11) & \cc & \aa & \pp & extended cogs for (3) and (7)\cr
(12) & \aa & \aa & \pp & extended cogs for (4) and (8)\cr
\hline
\end{tabular}
\endgroup
\bigskip
\caption{Subclasses of pseudocellular embeddings}\label{tab:subclass}
\end{table}

The subclasses (1), (2), (3), (4) of quasicellular embeddings correspond to Huggett and Moffatt's four models (in the same order) discussed in Subsection \ref{ss:prior}.  The results of \cite{EM15} and \cite[\S4.4 and \S4.8]{KMT18} can be regarded as results for stabilization classes of classic embeddings in pseudosurfaces, which is (4), quasicellular embeddings.

Subclass (5) corresponds to the \emph{edge-point ribbon graphs} defined by one of us (J.E.-M.), Kauffman and Moffatt \cite{EKM18}.  They constructed these as ribbon graphs where an edge may be contracted to a point, called a \emph{pinched-vertex}, identifying two points on vertex boundaries, or equivalently, an edge may be designated as a \emph{dark} edge.  However, dark edges may just be represented by strings.  Thus, edge-point ribbon graphs are equivalent to ribbon/string graphs, which correspond to cellular embeddings in surfaces where we can add cross-pinched edges.  In other words, we have embeddings in pseudosurfaces that are vertex- and face-cellular, but edges may be either cellular or cross-pinched, which is (5).  The operation of \emph{contracting an edge to a point} in \cite{EKM18} corresponds to what we have called BR pinching a cellular edge.

Using the correspondence established in the previous paragraph, subclass (9) corresponds to ribbon/string graphs in which all edges are strings.  As discussed in \cite[\S3.2.1]{EM13} or \cite{BEEMM25plus}, such objects correspond to \emph{cyclically ordered graphs} or \emph{cogs}, which are also known as \emph{rigid-vertex graphs}.  A cog is a graph equipped with an undirected cyclic ordering of the half-edges at each vertex.  Every cellular surface embedding has an associated cog, which identifies its equivalence class under partial Petrie duality.
As pseudocellular embeddings, cogs are represented by embeddings in pseudosurfaces where every pinchcomponent is a sphere that contains a single cellular vertex, which we will consider to be at the north pole of the sphere.
These spheres are identified into a pseudosurface at points of pinchdegree $2$, each of which represents the center of a cross-pinched edge, which we think of as being arranged around the equator of each sphere.  The half-edges on each sphere run from the vertex at the north pole to the points representing edge centers on the equator.
Each sphere contains exactly one cellular face of the embedding, and we may consider the dual vertex for this face to be at the south pole.  The dual half-edges run from the dual vertex at the south pole to the points representing edge centers on the equator.
For a cellular embedding in a surface (subclass (1)), the embedding corresponding to its associated cog may be found by BR pinching all of its edges (Q pinching all edges produces the cog, but with all faces in each component identified at a single pinchpoint).

\def\cC{\mathcal{C}}
For subclass (12) we extend (9) by dropping the requirements of vertex- and face-cellularity.  This means that we may identify north pole vertices together according to an arbitrary equivalence relation, and we may also identify south pole dual vertices together according to an arbitrary equivalence relation.
If $G$ is in subclass $(i)$ or $(i+4)$ for $i = 3,4$, then Q pinching all cellular edges of $G$ yields $G'$ in subclass $(i+8)$, which is a subclass of $(i+4)$.
If $G$ is in subclass $(i)$ or $(i+4)$ for $i = 1,2$, then BR pinching all cellular edges of $G$ yields $G'$ in subclass $(i+8)$, which is a subclass of $(i+4)$.
We can therefore consider elements of subclasses (10) to (12) as \emph{extended cogs}.

\section{Additional operations for pseudocellular embeddings}
\label{sec:addop}

In this section we define some further operations for pseudocellular embeddings that act on or eliminate edges, and eliminate faces and vertices.  The reader may wish to refer to Table \ref{tab:opsumm} in Section \ref{sec:conclusion} to get an overview of the operations we will cover.

\subsection{Twisted duality for pseudocellular embeddings}\label{ss:twdualpsemb}
As mentioned in Subsection \ref{ss:twdualsurf}, cellular embeddings of graphs in surfaces have three types of duality, and partial versions of these.  Here we briefly discuss some situations where we can implement partial dualities for pseudocellular embeddings.

We begin by considering Petrie duality, and specifically partial Petrie duality (twisting) for a single edge $e$ in a pseudocellular embedding $G$.
First consider a cellular edge $e$.  Twisting an edge $e$ of a cellular embedding in a surface can be interpreted topologically as adding a crosscap at the center of $e$, through which $e$ now passes, and then cellularizing faces if necessary.  The corresponding operation for a cellular edge $e$ of a pseudocellular embedding is as follows.  First add a crosscap  at the center of $e$, and make both $e$ and $e^*$ pass through the crosscap, so that they no longer intersect at the center of $e$.  This may create the same issues as when we remove an edge as the first step in an edge deletion, so we solve it in the same way, by crushing the dual edge. (We now get a new dual edge $e^*$, crossing $e$ at its center, in the new embedding.)  This is a \emph{Q edge twist} operation and we denote the result by $G \ppe e$.

For cross-pinched edges, since they do not transmit orientation information, it might make sense to define $G \ppe e = G$.  However, twisting a cellular edge has the side effect of merging the faces incident with the edge.  For consistency (justified further in Subsection \ref{ss:addop}) we therefore define the Q twist $G \ppe e$ for a cross-pinched edge $e$ as identification of the centers of the faces incident with $e$, which we refer to as \emph{coloopifying} $e$.

For all edges $e$, the underlying graph $\Ga$ of $G$ and $G \ppe e$ is the same.
For distinct edges $e_1$ and $e_2$, $G \ppe e_1 \ppe e_2 = G \ppe e_2 \ppe e_1$ (see Lemma \ref{lem:opcomm} below).  This means that $G \ppe A$ is well defined for sets of edges $A$.

The problem is that $G \mapsto G \ppe e$ is not in general an involution, considered as a map on all pseudocellular embeddings of graphs with a fixed edge set $E$.  This map never increases the number of faces, but sometimes decreases the number, so a decrease cannot be reversed.  In fact, this map is not even injective. Suppose $G$ is a cellular embedding in a surface, and $e$ is a colink with faces $f_1, f_2$ on its sides.  Twisting $e$ gives $G\ppe e$ in which $f_1$ and $f_2$ are merged into a single face $f$; i.e., the two dual vertices $f_1\du$ and $f_2\du$ become a single vertex $f\du$.
Now suppose $H$ is another cellular surface embedding with a face $g$.  Form two pseudocellular embeddings, constructing $G_1$ from $G$ and $H$ by identifying $f_1\du$ with $g\du$, and $G_2$ from $G$ and $H$ by identifying $f_2\du$ with $g\du$.  Then $G_1 \ppe e$ and $G_2 \ppe e$ are equal: both are obtained from $G \ppe e$ and $H$ by identifying $f\du$ with $g\du$.  Thus, the map is not injective.

However, if we are working with face-cellular graphs, we can perform a Q edge twist then dismantle any new facial pinchpoint, resulting in a \emph{BR edge twist} operation.  For cellular embeddings in surfaces, the BR edge twist is just the usual partial Petrie dual with respect to $e$.
The BR edge twist can be described for both cross-pinched and cellular edges using jewels.
For a face-cellular embedding we have a v-gon partition and a semiedge-digon partition, and the discrete f-gon partition (equivalently, no f-gon partition).  Twisting an edge $e$ means swapping colors $\cf$ and $\cz$ in the e-square or paired semiedge-digons corresponding to $e$.  The v-gon partition does not change.  For a Q edge twist we would merge the parts of the f-gon partition associated with $e$, but for the BR edge twist we then undo this by dismantling any facial pinchpoint that was created; in other words, the f-gon partition remains the discrete partition.
The set of semiedge-digons and the semiedge-digon partition remain the same.
For a cross-pinched edge, a BR edge twist does not change the embedding.
It is clear that the BR edge twist operation is an involution, so we can call it a partial Petrie dual.
As in the top part of Table \ref{tab:effpd}, this partial Petrie duality can both combine and split faces.
Thus, we can define partial Petrie duals $G \ppe A$ for sets of edges $A$ in face-cellular pseudocellular embeddings, which stay in this class.

Similarly, we can define a \emph{Q edge cotwist} operation by $G \pwi e = (G\du \ppe e^*)\du$, where $\ppe$ represents a Q edge twist.  
For cellular $e$ this corresponds to putting a crosscap at the center of $e$, through which $e$ and $e^*$ pass, crushing $e$, and replacing $e$ by the dual of $e^*$ in the new embedding.
For cross-pinched $e$ this corresponds to identifying the endvertices of $e$, which we refer to as \emph{loopifying} $e$.
Again this is not an involution: we can reduce the number vertices but not increase it.  However, if we are working with vertex-cellular embeddings we can perform a Q edge cotwist and then dismantle any new vertex pinchpoint, giving a \emph{BR edge cotwist} operation, which is an involution, and for cellular embeddings in surfaces it agrees with the partial Wilson dual.  So we can call this a partial Wilson dual.
As in the middle part of Table \ref{tab:effpd}, this partial Wilson duality can both combine and split vertices.
Thus, we can define partial Wilson duals $G \pwi A$ for sets of edges $A$ in vertex-cellular pseudocellular embeddings, which stay in this class.

If we work with pseudocellular embeddings that are both vertex- and face-cellular, then we can define BR partial geometric duals (Chmutov partial duals) by $G \pdu A = G \ppe \pwi \ppe A$, or equivalently $G \pwi \ppe \pwi A$, for sets of edges $A$ (when $A$ is a set of cross-pinched edges, $G \pdu A = G$).  This agrees with the usual partial duality for cellular embeddings in surfaces.  As in the bottom part of Table \ref{tab:effpd}, this partial duality can combine and split both vertices and faces, and is an involution.

Thus, for vertex-cellular face-cellular pseudocellular embeddings with edge set $E$ we have a complete $S_3^E$-action generated by the three partial duality operations.  This action corresponds to permuting the colors $\cv$, $\cf$ and $\cz$ in a jewel with a semiedge-digon partition.  These embeddings are cellular embeddings in surfaces with added cross-pinched edges.
As we discussed in Subsection \ref{ss:subclass}, these embeddings are equivalent to the edge-point ribbon graphs of one of us (J.E.-M.), Kauffman and Moffatt \cite{EKM18}.
In this setting we also have three minor operations (BR edge deletion, BR edge contraction, and BR twist-contraction (see Subsection \ref{ss:addop})), which interact with the partial dualities in the expected way.
This is a moderate extension to the twisted duality framework for cellular embeddings in surfaces.

We believe there are also other situations where we can find a complete $S_3^E$-action and appropriate minor operations, but we leave a complete investigation of this for future research.

\subsection{Further edge elimination operations}\label{ss:addop}
In this subsection we discuss some additional edge-elimination operations that we will use in constructing a family of polynomial invariants in Section \ref{sec:poly}.
Recall that \emph{elimination} is a non-specific term for any operation that gets rid of an edge, in contrast to specific operations such as deletion or contraction.

For cellular embeddings in surfaces we have three minor operations: deletion $\dt$, contraction $\ct$, and twist-contraction $\tc$, where $G \tc e = G \ppe e \ct e$, using BR operations.  For pseudocellular embeddings we define \emph{Q twist-contraction} by $G \tc e = G \ppe e \ct e$, using Q operations.
For cellular $e$ this means that we put a crosscap at the center of $e$, through which $e$ and $e^*$ now pass without intersecting, and crush both $e$ and $e^*$.
For cross-pinched $e$ we erase $e$ and identify the endvertices of $e$ and also identify the centers of the faces incident with $e$ (the endvertices of $e^*$).
Both of these descriptions are symmetric in $e$ and $e^*$, so $G \tc e$ is a self-dual operation.
We can interpret twist-contraction in a jewel with v-gon, f-gon and semiedge-digon partitions.
First smooth the subgraph $S_e$ (e-square or paired semiedge-digons) corresponding to $e$ along $\cz$.
Merge the parts of the v-gon partition containing the two (possibly equal) v-gons that used the $\cv$-colored edges of $S_e$, and also add the new v-gon(s) to this part.  Similarly, merge the parts of the f-gon partition containing the two (possibly equal) f-gons that used the $\cf$-colored edges of $S_e$, and also add the new f-gon(s) to this part.  In a class of vertex-cellular face-cellular embeddings we can perform BR twist-contraction in a jewel without v-gon or f-gon partitions by just smoothing $S_e$ along $\cz$.

We can now justify the definition we made in Subsection \ref{ss:twdualpsemb} for $G \ppe e$ when $e$ is a cross-pinched edge.  For a cellular edge $e_1$, our definition of $G \ppe e_1$ is essentially forced by generalizing the definition for cellular embeddings in surfaces and maintaining pseudocellularity, and then $G \tc e_1 = G \ppe e_1 \ct e_1$ turns out to be a self-dual operation.  However, for a cross-pinched edge $e_2$ we have some freedom in defining $G \ppe e_2$, which we resolve by requiring that $G \tc e_2 = G\ppe e_2 \ct e_2$ be self-dual as well.  The operation $G\tc e_2 = G \ppe e_2 \ct e_2$ identifies the endvertices of $e_2$, but if it is self-dual it should also identify the centers of the faces incident with $e_2$, which `$\ct e_2$' does not do.  Therefore, we need to define `$\ppe e_2$' for a cross-pinched edge $e_2$ so that it identifies the centers of the faces incident with $e_2$.  
The operation $G \tc e_2 = G \ppe e_2 \ct e_2$ is then self-dual, as required.
Self-duality gives $G \tc e = G \ppe e \ct e = G \pwi e \dt e$ for all edges $e$.

Our quotient-based edge contraction and deletion operations identify vertices or face centers only as needed to preserve pseudocellularity of an embedding $G$. But in some cases we can make additional optional identifications.
When we contract an edge $e$, we always identify its two endvertices, but we do not modify its incident faces.  We can optionally identify the centers of the two incident faces after contraction.  We call this operation \emph{supercontraction} and denote it by $G \sct e$.
Similarly, we can define the dual operation \emph{superdeletion}, denoted by $G \sdt e$, by identifying the endvertices of an edge after deleting it.
If $e$ is cellular, then $G\sct e = G \dpn e \dt e$, and $G \sdt e = G \pn e \ct e$.  
At this point we have five distinct operations for eliminating cellular edges: $G \ct e$, $G \dt e$, $G \tc e$, $G \sct e$, $G \sdt e$.

We also currently have four distinct operations for eliminating cross-pinched edges.
If $e$ is a cross-pinched edge, then $G \tc e = G \sct e = G \sdt e$: in all three cases we erase $e$ and then identify the two endvertices of $e$ and also identify the centers of the two faces incident with $e$.
So the four operations are $G \ct e$, $G \dt e$, $G\tc e$ ($=G \sct e =G\sdt e$), and $G \er e$.  We will define a fifth operation for cross-pinched edges, but our operation produces two possibly distinct embeddings, so to keep track of both of these simultaneously we work with formal $\mR$-linear combinations of edge-labeled pseudocellular embeddings.  It is easiest to define this operation using gems or jewels.
Suppose $e$ is a cross-pinched edge, and we have a gem or jewel with v-gon, f-gon and semiedge-digon partitions.
Then there is a pair of semiedge-digons $D^1_e$, $D^2_e$ corresponding to $e$, say one with vertices $u_1, u_2$ and the other with vertices $v_1, v_2$.  To form one of the two embeddings, delete all semiedge-digon edges between $u_1$ and $u_2$ and between $v_1$ and $v_2$, so that $u_1, u_2, v_1, v_2$ all have degree $1$.  Now add edges $u_1 v_1$ and $u_2 v_2$ of color $\cz$, and then smooth these edges along $\cz$.
Merge the parts of the v-gon partition containing the two (possibly equal) v-gons that used the $\cv$-colored edges of $D^1_e$ and $D^2_e$, and also add the new v-gon(s) to this part.  Similarly, merge the parts of the f-gon partition containing the two (possibly equal) f-gons that used the $\cf$-colored edges of $D^1_e$ and $D^2_e$, and also add the new f-gon(s) to this part.
The result represents a new pseudocellular embedding, say $G_1$.  But the labeling of $v_1$ and $v_2$ was arbitrary, so we could also perform this procedure using edges $u_1 v_2, u_2 v_1$, to obtain a second pseudocellular embedding $G_2$.  We define $G \sw e$, which we say is obtained by \emph{swirling} $e$, to be the linear combination $\frac12 G_1 + \frac12 G_2$.  This is a modified version of an operation used for cogs in \cite[Section 4]{BEEMM25plus}; our modified form normalizes by including the factor $\frac12$.

A small case analysis shows that the above ten operations, five for cellular edges and five for cross-pinched edges, are the only natural operations for eliminating an edge from a general pseudocellular embedding to produce another such embedding.  In terms of the jewel, a `natural' operation on $e$ should smooth two independent edges of the edge-colored subgraph $S_e$ (e-square or paired semiedge-digons) corresponding to $e$, where this pair of edges is invariant under symmetries of $S_e$.  Some mergers of parts of the v-gon or f-gon partitions may be forced by this, and we can optionally add unforced mergers of the parts that intersect $S_e$.  All ten operations except swirling fit this pattern.  If we allow adding edges not already in $S_e$ then we also obtain swirling, although to maintain invariance under symmetries of $S_e$ we add two separate pairs of independent edges, and average the results.

\begin{lemma}\label{lem:opcomm}
The ten edge elimination operations all commute, in the sense that 
we have $G \,o_1\, e_1 \,o_2\, e_2 = G \,o_2\, e_2 \,o_1\, e_1$ whenever $e_1, e_2$ are distinct edges and $o_1, o_2$ are (not necessarily distinct) operations such that $o_1$ can be applied to $e_1$ and $o_2$ can be applied to $e_2$.
\end{lemma}

\begin{proof}
We outline a proof.  Suppose $G$ has jewel $J$ with v-gon, f-gon, and semiedge-digon partitions.  We can track the edge elimination operations on a graph $J^+$ with $V(J^+)=V(J)$ by ignoring step (2) of each smoothing operation as defined in Subsection \ref{ss:repembsurf}.  Initially $J^+=J$, but we then replace the v-gon and f-gon partitions by using edges of new colors $\cvv, \cff$, where we put a connected subgraph of edges of color $\cvv$ on each subset of $V(J)$ corresponding to a part of the v-gon partition, and similarly with edges of color $\cff$ for the f-gon partition.  In $J^+$ a v-gon is a cycle whose edges are colored $\ca$ and $\cv$, but after performing some operations such a cycle may have more edges of color $\ca$ than of $\cv$.  A part of the v-gon partition corresponds to a component of the subgraph of $J^+$ induced by edges of colors $\ca$, $\cv$, and $\cvv$.  The f-gons and f-gon partition are defined similarly using edges of colors $\ca$, $\cf$, and $\cff$.  Now each operation involving an edge $e$ is implemented by replacing all original edges of $S_e$ of colors $\cv, \cf, \cz$ by two independent edges of color $\ca$ between vertices of $S_e$, and any required mergers of parts of the v-gon or f-gon partition can be implemented by adding edges of color $\cvv$ or $\cff$, respectively, between vertices of $S_e$.  Each operation therefore only affects the subgraph induced by $V(S_e)$.  Thus, operations on distinct edges $e_1$ and $e_2$ commute.  (We have ignored complications due to swirling, but those can be handled.)
\end{proof}

We will use these ten operations to define a family of polynomial invariants in the next section.

\subsection{Vertex and face elimination and edgeless embeddings}\label{ss:edgeless}%
We now discuss ways to eliminate vertices and faces from a pseudocellular embedding.
In addition to edge deletion and contraction, vertex deletion is also considered a minor operation for abstract graphs.
Vertex deletion really only needs to be applied to isolated vertices (with no incident edges), as deletion of a general vertex $v$ can be implemented by deleting all edges incident with $v$ and then deleting the now-isolated $v$.
To delete a vertex $v$ in a ribbon graph we can remove all edges incident with $v$, leaving $v$ represented by an isolated vertex disk with one boundary component, representing a single face $f$ incident with $v$, and then remove the disk for $v$, implicitly also removing $f$.
For a cellular embedding $G$ of a graph in a surface, this corresponds to removing $v$ and all of its incident edges from the embedding, and then cellularizing; the result is denoted by $G \dt v$.

For ribbon graphs or cellular embeddings in surfaces there is also an operation dual to vertex deletion, which we call \emph{face contraction}: for a face $f$ we define $G \ct f = (G\du \dt f\du)\du$.
For a cellular embedding of a graph in a surface we can implement this as follows.  First contract all edges incident with $f$: this is equivalent to crushing the boundary $\bdy f$ (all vertices and edges with $f$) to a vertex $v$.  The only possible pinchpoint is $v$ and we dismantle $v$ to form a surface embedding.  Now $f$ is an \emph{isolated face}, meaning that its boundary contains no edges (equivalently, there are no dual edges incident with $f\du$).  So there is a component containing $f$ and a copy of $v$, and we discard this component.  This is equivalent to crushing the closure of the face $f$ to a single point $v$, and either discarding $v$ if it is a single-point component of the resulting space, or dismantling $v$ otherwise.

For pseudocellular embeddings the situation is more complicated: there are two ways to eliminate a vertex $v$, a weak way and a strong way, which are the same if $v$ is cellular.  The weak operation leaves more of the structure of the original embedding intact than the strong operation.  Similarly, there are two ways to eliminate a face $f$, which are the same if $f$ is cellular.  In both cases we will call the weak elimination \emph{deletion} and the strong elimination \emph{contraction}: vertex deletion is dual to face deletion, and vertex contraction is dual to face contraction.  This seems odd, but the terminology agrees with intuitive ideas of `deletion' as removal and `contraction' as crushing.

Suppose now that $f$ is an isolated face, so that the boundary of each $P_i$ in $\Pi$ is a vertex $v_i$ (these vertices are not necessarily distinct).  For the \emph{face deletion} of $f$, we take $\Pi-f$, and discard any points $v_i$ that are single-point components of $\Pi-f$; the result is denoted $G \dt f$.
Single-point components must be deleted or we no longer have a pseudosurface.
For the \emph{face contraction} of $f$, we take $G \dt f$, and if any points $v_i$ remain we identify them all to a single point $f\du$; the result is denoted $G \ct f$.  This is equivalent to crushing the closure of $f$ to the single point $f\du$, and discarding $f\du$ if it is a single-point component of the resulting space.  If $f$ is not an isolated face, for either $G \dt f$ or $G \ct f$ we first contract all edges incident with $f$, making $f$ an isolated face, and then perform the appropriate operation described above.  For a vertex $v$, we then define $G\dt v$ as $(G\du \dt v\du)\du$ and $G \ct v$ as $(G\du \ct v\du)\du$.  For both vertices and faces, contraction potentially merges more points than deletion.

Pseudocellular embeddings have nontrivial structure even if they have no edges, and we can apply the above operations to edgeless pseudocellular embeddings.  The structure is represented by a \emph{bipartitioned graph}, meaning a bipartite graph with a specific ordered bipartition.  Here we briefly describe some ideas due to the first two authors as extensions of \cite[Chapter 6]{Dun20}.
A pseudocellular embedding $G$ of an edgeless graph $\Ga = \overline{K_n}$ has an underlying pseudosurface $\Pi$ where each pinchcomponent is a sphere $S$ that contains one vertex $v$ and one face center $f^*$; $S-\{v\}$ is a plate of the face $f$.  Thus, the structure can be represented by a bipartitioned graph $H$ (possibly with parallel edges), where $V(H)$ is partitioned into $V(G)$ and $F^*(G)$ (the set of face centers of $G$), with one edge between $v$ and $f^*$ for each pinchcomponent of $\Pi$ that contains both $v$ and $f^*$.
Thus, $G$ can be thought of as a `fattened' version of $H$, with each edge replaced by a sphere.  Every bipartitioned graph $H$ with no isolated vertices corresponds to an embedding $G$.

Deleting a face or a vertex in $G$ then just corresponds to ordinary vertex deletion of the corresponding vertex of $H$, except that any isolated vertices created are discarded.  Contracting a face or vertex in $G$ corresponds to a \emph{vertex contraction} of the corresponding vertex $w$ of $H$, where we contract all edges of $H$ incident with $w$ into $w$ (keeping any parallel edges that are created).  We then discard $w$ if it is isolated.  So performing vertex and face deletions and contractions for edgeless pseudocellular embeddings is equivalent to performing operations on the class of bipartitioned graphs with no isolated vertices.

This suggests the investigation of vertex deletion and vertex contraction in the context of abstract graphs, which gives an apparently new containment relation for graphs with some interesting properties.  We modify the setting a little. 
We do not require graphs to be bipartite.  We allow graphs with isolated vertices, so we do not remove isolated vertices created by a vertex deletion or contraction.  But we restrict to simple graphs, so after performing a vertex contraction we remove loops, and reduce classes of parallel edges to single edges.
A \emph{vertex contraction-deletion minor}, or \emph{VCD minor}, of a simple graph $G$ is a simple graph $H$ obtained from $G$ by a sequence of such vertex deletions and vertex contractions.

The class of (simple) bipartite graphs is closed under taking VCD minors, and bipartite graphs are the $C_3$-VCD-minor-free graphs.
VCD minors have been investigated by the second author in collaboration with Sarah Allred and Gabri\"{e}lle Zwaneveld \cite{AE25+,AEZ26+}.
Other important classes of graphs, including chordal graphs and chordal bipartite graphs, can also be described by a finite set of forbidden VCD minors \cite{AE25+}.
VCD minors are related to \emph{t-minors}, introduced in \cite{GS98} and defined in \cite{BS10}, which preserve a property known as \emph{t-perfection}.  For t-minors we allow vertex deletion, but vertex contraction is only allowed for vertices whose neighbors are independent.  Thus, for bipartite graphs VCD minors are the same as t-minors.

\section{A family of polynomial invariants for pseudocellular embeddings}
\label{sec:poly}

In this section we discuss a family of polynomial invariants for pseudocellular embeddings.

One of us (J.E.-M.) and Moffatt \cite[\S5]{EM12} introduced a \emph{topological transition polynomial} for cellular embeddings of graphs in surfaces, which uses the three minor operations of (BR) edge deletion, edge contraction, and edge twist-contraction.  This was defined in a general way, using a weight for each individual edge and way that the edge could be reduced.  If we assume that all edge contractions have weight $\al$, all edge deletions have weight $\be$, and all edge twist-contractions have weight $\ga$, we obtain a recursive definition (where we abbreviate $Q(G, (\al, \be, \ga), t)$ to just $Q(G)$)
\begin{equation}\label{eq:transrecurs}
Q(G) = \begin{cases}
	t^{c(G)} &\text{if $G$ has no edges,} \\
	\al Q(G\ct e) + \be Q(G \dt e) + \ga Q(G \tc e)
		&\text{for an arbitrary edge $e$, otherwise,} \\
\end{cases}
\end{equation}
where $c(G)$ denotes the number of components of the embedded graph $G$.  This is equivalent to the following definition, where the sum is over all ordered partitions $(A, B, C)$ of $E(G)$.
\begin{equation}\label{eq:trans}
Q(G, (\al, \be, \ga), t) = \sum_{(A,B,C)} \al^{|A|}\, \be^{|B|}\, \ga^{|C|}\, t^{c(G \ct A \dt B \ms1\tc\ms2 C)} .
\end{equation}
The equivalence of the two formulas relies on the fact that the three minor operations commute, in the sense that we have $G \,o_1\, e_1 \,o_2\, e_2 = G \,o_2\, e_2 \,o_1\, e_1$ whenever $e_1, e_2$ are distinct edges and $o_1, o_2$ are (not necessarily distinct) operations.  Thus, the order in which we consider edges in the recursive formula (\ref{eq:transrecurs}) does not matter, and we can just consider the sets of edges on which we performed each of the operations, giving (\ref{eq:trans}).

\def\bc{_{\fam0 c}}
\def\bd{_{\fam0 d}}
\def\bsc{_{\fam0 sc}}
\def\bsd{_{\fam0 sd}}
\def\btc{_{\fam0 t}}
\def\ber{_{\fam0 e}}
\def\bsw{_{\fam0 w}}
\def\bv{_{\fam0 v}}
\def\bf{_{\fam0 f}}
\def\bk{_{\fam0 k}}
\def\bp{_{\fam0 p}}

For pseudocellular embeddings we have ten separate operations we can use to eliminate edges from an embedding, as discussed in Subsection \ref{ss:addop}.  By Lemma \ref{lem:opcomm}, they all commute in the same sense as the three minor operations discussed above.  Thus, we can unambiguously apply each of these operations to sets of edges, not just to individual edges.

We can use a recursive formula similar to (\ref{eq:transrecurs}) to define a \emph{transition-merger polynomial} $\Psi(G) = \Psi(G, (x\bc, x\bd, x\btc, x\bsc, x\bsd;\allowbreak y\bc, y\bd, y\btc, y\ber, y\bsw), \psi)$ by the recurrence relation
\begin{equation}\label{eq:tmrecurs}
\Psi(G) = \begin{cases}
	\psi(G) &\text{if $G$ has no edges, or} \\
	x\bc \Psi(G\ct e) + x\bd \Psi(G \dt e) + x\btc \Psi(G \tc e) + x\bsc \Psi(G \sct e) +
		x\bsd \Psi(G \sdt e) \mskip-300mu \\
	  &\text{for an arbitrary cellular edge $e$, or\qquad} \\
	y\bc \Psi(G\ct e) + y\bd \Psi(G \dt e) + y\btc \Psi(G \tc e) +
		y\ber \Psi(G \er e) + y\bsw \Psi(G \sw e)\mskip-350mu \\
	  &\text{for an arbitrary cross-pinched edge $e$,} \\
\end{cases}
\end{equation}
where $\psi$ is any invariant of edgeless pseudocellular embeddings, and all operations are quotient-based (Q) operations. 
While this formula defines the value for an individual embedding $G$, we need to work in the set of formal $\mR$-linear combinations of edge-labeled pseudocellular embeddings because we need that framework to define $G \sw e$.  For such a linear combination $\sum_i \al_i G_i$, we define $\Psi(\sum_i \al_i G_i)$ by linearity, as $\sum_i \al_i \Psi(G_i)$.  We also need to apply $\psi$ to $\mR$-linear combinations of edgeless embeddings, and again we extend $\psi$ to this setting by linearity.

Using the approach of Subsection \ref{ss:edgeless}, we may consider $\psi$ to be an invariant of bipartite graphs $H$ whose vertices are properly $2$-colored with either $\cv$ (indicating a vertex) or $\cf$ (indicating a face).  A simple choice might be $\psi(G) = t\bv^{v(G)}\, t\bf^{f(G)}\, t\bk^{k(G)}\, t\bp^{p(G)}$ where $v, f, k, p$ represent the numbers of vertices, faces, pseudosurface components, and pinchcomponents (or equivalently face plates) of the edgeless embedding $G$, which are, respectively, the numbers of vertices colored $\cv$, vertices colored $\cf$, components, and edges of the bipartite graph $H$.

As with (\ref{eq:transrecurs}), we can obtain a formula equivalent to (\ref{eq:tmrecurs}) as a single sum.  We sum over ordered pairs $(\mathcal{A}, \mathcal{B})$ where $\mathcal{A} = (A\bc, A\bd, A\btc, A\bsc, A\bsd)$ is an ordered partition of the cellular edges of $G$, and $\mathcal{B} = (B\bc, B\bd, B\btc, B\ber, B\bsw)$ is an ordered partition of the cross-pinched edges of $G$.  We have
\begin{equation}\begin{array}{rl}
\Psi(G, &\ms{-10}(x\bc, x\bd, x\btc, x\bsc, x\bsd; y\bc, y\bd, y\btc, y\ber, y\bsw), \psi) \\
&= \displaystyle \sum_{(\mathcal{A},\mathcal{B})} \big(\,
x\bc^{|A\bc|}\,
x\bd^{|A\bd|}\,
x\btc^{|A\btc|}\,
x\bsc^{|A\bsc|}\,
x\bsd^{|A\bsd|}\,
y\bc^{|B\bc|}\,
y\bd^{|B\bd|}\,
y\btc^{|B\btc|}\,
y\ber^{|B\ber|}\,
y\bsw^{|B\bsw|}\, \\
&\qquad\qquad\qquad
	\cdot\,\psi(G \ct A\bc \dt A\bd \tc A\btc \sct A\bsc \sdt A\bsd
	\ct B\bc \dt B\bd \tc B\btc \er B\ber \sw B\bsw) \,\big). \\
\end{array}\end{equation}

\begin{observation}
For cellular edges, $\ct$ and $\dt$ are dual operations, and $\sct$ and $\sdt$ are dual operations, while $\tc$ is self-dual.  For cross-pinched edges, $\ct$ and $\dt$ are dual operations, while $\tc$, $\er$, and $\sw$ are self-dual.  Therefore, we have
\begin{equation}\begin{array}{rl}
\Psi(G, &\ms{-10}(x\bc, x\bd, x\btc, x\bsc, x\bsd; y\bc, y\bd, y\btc, y\ber, y\bsw), \psi) \\[4pt]
&= \Psi(G\du, (x\bd, x\bc, x\btc, x\bsd, x\bsc; y\bd, y\bc, y\btc, y\ber, y\bsw), \psi\du) \\
\end{array}\end{equation}
where $\psi\du(F) = \psi(F\du)$ for edgeless pseudocellular embeddings $F$.
\end{observation}

We can adapt $\Psi(G)$ for use in classes of embeddings where we require vertices and/or faces to be cellular by using BR modifications of the operations for eliminating edges.  As mentioned earlier, the general rule is that a BR operation is derived from a Q operation by dismantling any new pinchpoints formed by an operation that violate required cellularity conditions.  However, in this case some of the operations become equivalent.  For example, contraction and supercontraction become equivalent if we want to maintain face-cellularity, because the face center identification that occurs in supercontraction but not contraction is undone by the dismantling step.

In some cases we can use the original $\Psi(G)$, based on Q operations, to obtain results for situations where BR operations are used, as the following indicates.

\begin{observation}\label{obs:QPsi} (Due to ChatGPT Astra)
Although the transition polynomial $Q(G)$ for cellular surface embeddings is based on BR minor operations and $\Psi(G)$ is based on Q operations, it is possible to obtain $Q(G)$ as a specialization of $\Psi(G)$.  The pinchpoint dismantlings involved in BR minor operations can all be postponed until we have an edgeless embedding.
So suppose we perform Q minor operations on a cellular surface embedding $G$ to obtain an edgeless pseudocellular embedding $G'$, then dismantle all pinchpoints to obtain an edgeless cellular surface embedding $G''$.  We could also have obtained $G''$ directly from $G$ by applying the corresponding sequence of BR minor operations.
The pinchcomponents of $G'$ become the components of $G''$, so if we use the simple function $\psi\bp(G) = t^{p(G)}$, then $\psi\bp(G') = t^{p(G')} = t^{c(G'')}$.  Therefore, for a cellular surface embedding $G$, we have
\begin{equation}
Q(G, (\al,\be,\ga), t) = \Psi(G, (\al,\be,\ga,0,0;0,0,0,0,0),\psi\bp).
\end{equation}
\end{observation}

\section{Conclusion} \label{sec:conclusion}

\def\zz{  \vrule height12pt depth3pt width0pt}

\begin{table}[ht]
\begin{tabular}{|l|c|c|c|l|}
  \hline
  \zz Name & Notation & Defined & See also & Comments \\
  \hline
  \zz Edge contraction & $G \ct e$ & \S\ref{ss:pseudocellular} & \S\ref{ss:minorsurf}, \S\ref{ss:qbr}
        & dual to $G \dt e$ \\
  \zz Edge deletion & $G \dt e$ & \S\ref{ss:pseudocellular} & \S\ref{ss:minorsurf}, \S\ref{ss:qbr}
        & dual to $G \ct e$ \\
  \zz Edge twist-contraction & $G \tc e$ & \S\ref{ss:addop} & \S\ref{ss:minorsurf}
        & $G \ppe e \ct e$; self-dual \\
  \zz Edge supercontraction & $G \sct e$ & \S\ref{ss:addop} & & dual to $G \sdt e$ \\
  \zz Edge superdeletion & $G \sdt e$ & \S\ref{ss:addop} & & dual to $G \sct e$ \\
  \zz Edge erasure & $G \er e$ & \S\ref{ss:pseudocellular} & & cross-pinched $e$ only; self-dual \\
  \zz Edge swirl & $G \sw e$ & \S\ref{ss:addop} & & cross-pinched $e$ only; self-dual \\
  \hline
  \zz Edge pinch & $G \pn e$ & \S\ref{ss:pseudocellular} & & cellular $e$ only; dual to $G \dpn e$ \\
  \zz Edge copinch & $G \dpn e$ & \S\ref{ss:pseudocellular} & & cellular $e$ only; dual to $G \pn e$ \\
  \hline
  \zz Edge twist & $G \ppe e$ & \S\ref{ss:twdualpsemb} & \S\ref{ss:twdualsurf}, \S\ref{ss:addop} & dual to $G \pwi e$ \\
  \zz Edge cotwist & $G \pwi e$ & \S\ref{ss:twdualpsemb} & \S\ref{ss:twdualsurf} & dual to $G \ppe e$ \\
  \zz Edge partial dual & $G \pdu e$ & \S\ref{ss:twdualpsemb} & \S\ref{ss:twdualsurf} & restricted $G$; self-dual \\
  \hline
  \zz Face contraction & $G \ct f$ & \S\ref{ss:edgeless} & & dual to $G \ct v$ \\
  \zz Face deletion & $G \dt f$ & \S\ref{ss:edgeless} & & dual to $G \dt v$ \\
  \hline
  \zz Vertex contraction & $G \ct v$ & \S\ref{ss:edgeless} & & dual to $G \ct f$ \\
  \zz Vertex deletion & $G \dt v$ & \S\ref{ss:edgeless} & \S\ref{ss:minorsurf}
        & dual to $G \dt f$ \\
  \hline
\end{tabular}
\bigskip
\caption{Summary of operations on pseudocellular embeddings}\label{tab:opsumm}
\end{table}

For the reader's convenience, a table summarizing all pseudocellular embedding operations discussed in this paper appears in Table \ref{tab:opsumm}.
There are many further directions in which pseudocellular embeddings could be investigated.  We mention some relevant questions below.

\begin{enumerate}[(1),parsep=2pt]
\item In Subsection \ref{ss:twdualpsemb} we discuss one particular situation where we have twisted duality for pseudocellular embeddings, namely when we have vertex-cellular face-cellular embeddings.  This is based on BR edge twist (partial Petrie dual) and BR edge cotwist (partial Wilson dual) operations, where we preserve vertex- and face-cellularity. Can we find a subclass of pseudocellular embeddings for which the Q edge twist and Q edge cotwist operations give a twisted duality framework?

\item We defined pinching and copinching operations, and there are directions in which these can be investigated further.  How should we define `unpinching' operations that turn a cross-pinched edge into a cellular edge?
We discussed turning quasicellular embeddings into extended cogs by pinching all edges, but what happens if we use a combination of pinching and copinching to make all edges cross-pinched?

\item Embedded medial graphs play an important role in the theory of graphs in surfaces, and have been defined in some more general situations.  For example, medial graphs of edge-point ribbon graphs are defined in \cite{EKM18}.  How should embedded medial graphs be constructed for pseudocellular embeddings, and how do their properties relate to properties of the embedding?

\item Edmonds \cite{Edm65} gave a theorem that characterizes when two abstract graphs can be the underlying graphs of two cellular embeddings in surfaces that are dual to each other.  Can we extend this to quasicellular embeddings and to pseudocellular embeddings?  For pseudocellular embeddings, we may want to specify which edges should be cellular and which should be cross-pinched.

\item As mentioned in Subsection \ref{ss:prior}, Moffatt and Thompson \cite{MT24,Thom26} defined a Tutte polynomial with many nice properties for packaged ribbon graphs, which are equivalent to quasicellular embeddings with nonnegative integer weights for vertices and faces.  Can this be extended to weighted pseudocellular embeddings in a natural way?

\item A number of polynomials have been defined for objects that we have shown can be represented as pseudocellular embeddings, such as edge-point ribbon graphs \cite{EKM18} and cogs \cite{BEEMM25plus}.  Can these polynomials be expressed in natural way in terms of pseudocellular embeddings, and perhaps extended to larger classes of pseudocellular embeddings?

\item What can we say about embeddings of hypergraphs in pseudosurfaces?  It may be helpful to use the representation of hypergraph embeddings in surfaces as properly $3$-edge-colored graphs, similar to gems, which is due to Chmutov and Vignes-Tourneret \cite{CV22}.
\end{enumerate}

\section*{Acknowledgments}

M.E. thanks the Simons Foundation for support under awards 429625 and MPS-TSM-00002760.
M.E. also thanks Spencer Dowdall for topological assistance, and Chun-Hung Liu and Youngho Yoo for pointing out the connection between VCD minors and t-minors.

M.E. and J.E.-M. thank Iain Moffatt for pointing out the connection to edge-point ribbon graphs, and for many discussions on the possibility of partial duality for pseudosurface embeddings, which pushed us to think harder about these embeddings.
These discussions occurred at the \emph{Workshop on Uniqueness and Discernment in Graph Polynomials} held at the MATRIX Institute for the Mathematical Sciences in October 2023.
M.E. and J.E.-M. are very grateful for MATRIX--Simons Travel Grants which partially supported their participation in this workshop, and for the productive research environment provided by MATRIX.
We also thank our coauthors on \cite{BEEMM25plus}, namely Iain Moffatt, Paul Bratch, and Wout Moltmaker, as some of the ideas in this paper build on that one.

AI systems including various levels of Claude and ChatGPT were used to help create figures based on hand-drawn originals, and to proofread this paper during preparation.  The AI systems provided useful suggestions and corrections, including one new observation (Observation \ref{obs:QPsi}).  All writing was done by the authors and the authors take responsibility for all statements in this paper.

\bibliography{e}
\bibliographystyle{amsplain}

\end{document}